\documentclass[12pt]{amsart}
\pdfoutput=1
\usepackage[numbers]{natbib}
\usepackage{hyperref}
\hypersetup{
     colorlinks   = true,
     citecolor    = red,
     linkcolor    = blue
}
\usepackage[foot]{amsaddr}
\usepackage[a4paper]{geometry}
\usepackage{setspace}
\usepackage{fancyhdr}
\usepackage{url}
\usepackage{amsmath}
\usepackage{amsfonts}
\usepackage{amssymb}
\usepackage{amsthm}
\usepackage{mathrsfs}
\usepackage{bbm, dsfont}
\usepackage{enumerate}%
\usepackage{xy,color}
\usepackage{graphicx}
\usepackage{tikz}
\usepackage{xy}
\usepackage{tikz-cd}

\theoremstyle{plain}
\newtheorem{theorem}{Theorem}[section]
\newtheorem{main}{Main Theorem}[section]
\newtheorem{corollary}[theorem]{Corollary}
\newtheorem{lemma}[theorem]{Lemma}
\newtheorem{proposition}[theorem]{Proposition}
\newtheorem*{ac}{Acknowledgement}

\newtheorem{notation}[theorem]{Notation}

\newtheorem{remark}[theorem]{Remark}

\newtheorem{definition}[theorem]{Definition}
\numberwithin{equation}{section}

\newcommand{\VeryThinLine}{\draw[line width=0.5pt]}
\newcommand{\ThinLine}{\draw[line width=1pt]}

\newcommand{\FillGrey}{\path[fill=gray!40!white]}
\newcommand{\FillWhite}{\path[fill=white]}

\newcommand{\Ch}{\mathbf{UCP}}
\newcommand{\id}{\text{Id}}
\newcommand{\cM}{\mathcal{M}}
\newcommand{\cN}{\mathcal{N}}
\newcommand{\End}{\text{End}}
\newcommand{\Hom}{\text{Hom}}

\begin{document}
\bibliographystyle{plain}
\title{Relative Entropy for Quantum Channels}
\author{Zishuo Zhao}
\address{Yau Mathematical Science Center and Department of Mathematics, Tsinghua University, Beijing, China}
\email{zzs21@mails.tsinghua.edu.cn}
\date{}
\begin{abstract}
    We introduce an quantum entropy for bimodule quantum channels on finite von Neumann algebras, generalizing the remarkable Pimsner-Popa entropy. 
    The relative entropy for Fourier multipliers of bimodule quantum channels establishes an upper bound of the  quantum entropy. 
    Additionally, we present the Araki relative entropy for bimodule quantum channels, revealing its equivalence to the relative entropy for Fourier multipliers and demonstrating its left/right monotonicities and convexity. 
    Notably, the quantum entropy attains its maximum if there is a downward Jones basic construction. 
    By considering R\'{e}nyi entropy for Fourier multipliers, we find a continuous bridge between the logarithm of the Pimsner-Popa index and the Pimsner-Popa entropy. 
    As a consequence, the R\'{e}nyi entropy at $1/2$ serves a criterion for the existence of a downward Jones basic construction. 
\end{abstract}
\maketitle
\iffalse
We introduce a quantum entropy for bimodule quantum channels on finite von Neumann algebras, extending the remarkable Pimsner-Popa entropy. The relative entropy for Fourier multipliers of bimodule quantum channels establishes an upper bound for the quantum entropy. Additionally, we present the Araki relative entropy for bimodule quantum channels, demonstrating its equivalence to the relative entropy for Fourier multipliers. Notably, the quantum entropy reaches its maximum when a downward Jones basic construction exists. Interestingly, the Rényi entropy for Fourier multipliers serves as a continuous link between the logarithm of the Pimsner-Popa index and the Pimsner-Popa entropy. Consequently, the Rényi entropy at 1/2 serves as a criterion for the existence of a Jones downward basic construction.
\fi

\section{Introduction}

% Relative entropy is a fundamental concept in many fields of sciences such as information theory, statistical mechanics, quantum thermodynamics etc.
% It is also important in harmonic analysis, functional analysis etc.
% In information theory, the relative entropy is known as Kullback-Leibler (KL) divergence or information divergence, which is defined for probability distributions.
% In quantum information theory, the relative entropy is defined for density matrices.
% In 1975, Araki \cite{Araki1977} introduced the relative entropy for states by means of Tomita-Takesaki theory in operator algebras.
% Araki's relative entropy generalized the relative entropy for density matrices.

Relative entropy, introduced independently by Kullback and Leibler, serves as a measure quantifying the disparity between two probability distributions. 
Umegaki \cite{Ume62} expanded the concept of relative entropy to encompass density matrices within quantum systems. Building on this foundation, Connes and St{\o}mer \cite{CS1975} delved into the study of relative entropy for subalgebras.
In a pivotal contribution, Pimsner and Popa investigated relative entropy for finite von Neumann algebras in their work \cite{PP1986}, coining the term "Pimsner-Popa entropy." Their study established a profound connection, demonstrating that the finiteness of the Jones index hinges on the satisfaction of the Pimsner-Popa inequalities. Notably, they showed that the Jones index is finite if and only if these inequalities hold, equivalently indicating the finiteness of the Pimsner-Popa entropy.

Quantum relative entropy and quantum channels assume pivotal roles in the exploration of quantum information theory. A bimodule quantum channel, denoted as $\Phi:\mathcal{M}\to\mathcal{M}$, is characterized by its preservation of a *-subalgebra $\mathcal{N}$.

In our pursuit of advancing the understanding of bimodule quantum channels, we propose several relative entropies. Drawing inspiration from the foundational work of Connes-St{\o}mer and Pimsner-Popa, we introduce the Pimsner-Popa entropy $H(\Phi|\Psi)$ tailored for bimodule quantum channels $\Phi$ and $\Psi$.

Employing the framework of quantum Fourier analysis (\cite{JLW2016}, \cite{JJLRW2020}, \cite{HLW2022}), we define the relative entropy for bimodule quantum channels as the quantum relative entropy $D(\widehat{\Phi}\|\widehat{\Psi})$ of the Fourier multipliers $\widehat{\Phi}$ and $\widehat{\Psi}$ which determine information of $\Phi, \Psi$ completely.
Our subsequent exploration aims to demonstrate that
 \begin{align*}
 H(\Phi|\Psi)\leq D(\widehat{\Phi}\|\widehat{\Psi}).
 \end{align*}
With insights from the spin model, conventional quantum channels operating on finite quantum systems manifest as bimodule quantum channels. In a surprising turn of events, we have substantiated that when the inclusion facilitates a downward Jones basic construction, it follows that
\begin{align*}
H(\Phi|\Psi)= D(\widehat{\Phi}\|\widehat{\Psi})
\end{align*}

In the realm of infinite quantum systems, Araki conducted a systematic exploration of relative entropy for normal states, extending the notion from finite quantum systems. Drawing inspiration from Araki's pioneering work, we introduce a relative entropy, denoted as $S_\tau(\Phi, \Psi)$, tailored for (bimodule) quantum channels $\Phi$ and $\Psi$. 
Our investigation reveals that this entropy exhibits both left and right monotonicity. 
Consequently, we obtain the convexity.
In contrast to the Pimsner-Popa entropy designed for bimodule quantum channels, we observe that
    \begin{align*}
        H(\Phi|\Psi)\leq S_{\tau}(\Phi,\Psi). 
    \end{align*}

Inspired by the study of the relation between Pimsner-Popa index and R\'{e}nyi entropy for subalgebras in \cite{GJL20}, we consider the R\'{e}nyi entropy between Fourier multipliers. 
We found that when a downward Jones basic construction exists, R\'{e}nyi entropy $S_p(\Phi|\Psi), \ p\in [1, \infty]$ for Fourier multipliers forms a continuous bridge between the logarithm of Pimsner-Popa index $\lambda(\Phi,\Psi) $ and Pimsner-Popa entropy $ H(\Phi|\Psi)$ for bimodule quantum channels, enhancing the result in \cite{GJL20}. 
    \begin{align*}
        -\log \lambda(\Phi,\Psi) \geq S_p(\Phi|\Psi) \geq H(\Phi|\Psi).
    \end{align*}
By comparing R\'{e}nyi entropy at $1/2$ and Pimsner-Popa entropy, we obtain a criterion for the existence of downward basic constructions when $\cM$ is a finite factor.

The paper is organized as follows.
In Section 2, we review finite inclusion of finite von Neumann algebras, completely positive maps and completely positive bimodule maps.
In Section 3, we introduce Pimsner-Popa relative entropy for completely positive maps.
In Section 4, we show the equality between Pimsner-Popa relative entropy and relative entropy for Fourier multiplier when the inclusion admits a downward basic construction.
In Section 5, we study the comparable completely positive maps and their derivatives.
In Section 6, we introduce Araki's relative entropy for comparable completely positive maps and prove its monotoniciy and convexity.
In Section 7, we study the R\'{e}nyi relative entropies for completely positive maps.

\begin{ac}
    The author would like to express his gratitude to Zhengwei Liu and Jinsong Wu for their constant support and encouragement. 
    The author was supported by BMSTC and ACZSP (Grant No. Z221100002722017) and by Beijing Natural Science Foundation Key Program (Grant No. Z220002).
\end{ac}

\section{Preliminaries}

In this section, we review the basic theory of inclusions of finite von Neumann algebras and their completely positive bimodule maps.

\subsection{Jones basic construction}

Let $\cM$ be a finite von Neumann algebra with a faithful normal normalized trace $\tau_{\cM}$, and let $\cN\subset \cM$ be an inclusion. 
We denote the restriction of $\tau_{\cM}$ on $\cN$ as $\tau_{\cN}$, and let $E_{\cN} = E^{\cM}_{\cN}$ be the trace-preserving conditional expectation from $\cM$ onto $\cN$. 
We fix a set of Pimnser-Popa basis $\{\eta_j\}_{j}$ for the right $\cN$-module $L^2(\cM)_{\cN}$. 
That is, $\sum_j \eta_jE_{\cN}(\eta^*_jx) = x$ for all $x\in\cM$. 
Consider the operators $L_{\tau_{\cN}}(\eta_j)\in \Hom(L^2(\cN)_{\cN},L^2(\cM)_{\cN})$ defined as 
\begin{align*}
    L_{\tau_{\cN}}(\eta_j)(y\Omega_{\cN}) = \eta_jy\Omega_{\cM},\quad y\in\cN.
\end{align*}
Then $\{\eta_j\}_{j}$ being a basis implies that 
\begin{align*}
    \sum_j L_{\tau_{\cN}}(\eta_j)L^*_{\tau_{\cN}}(\eta_j) = 1
\end{align*}
as operators on $L^2(\cM)$. 
We say $\cN\subset\cM$ is a \textbf{finite inclusion} if there exists a finite Pimsner-Popa basis. 
In such case, the \textbf{Jones index} of the inclusion $\cN\subset\cM$ is defined to be $\delta^2 = [\cM:\cN]:=\dim_{\cN}L^2(\cM)$ and can be computed as 
\begin{equation}\label{eqn::value of the index}
    \delta^2 = \sum_j\tau_{\cM}(\eta^*_j\eta_j).
\end{equation}

Let $J_{\cM}$ be the modular conjugation on $L^2(\cM)$ associated to $\tau_{\cM}$ and $e_{\cN}$ be the orthogonal projection with range $\overline{\cN\Omega_{\cM}}$. 
The Jones basic construction for $\cN\subset\cM$ is defined as $\cM_1 = J_{\cM}\cN'J_{\cM} = \cM e_{\cN}\cM$, with a canonical trace 
\begin{equation}\label{eqn::canonical trace on M1}
    \tau_{\cM_1} (z) = \delta^{-2}\sum_j\langle z(\eta_j\Omega_{\cM}),\eta_j\Omega_{\cM}\rangle,\quad z\in \cM_1,
\end{equation}
and the canonical trace on $\cN' = J_{\cM}\cM_1J_{\cM}$ is defined as 
\begin{equation}\label{eqn::canonical trace on N'}
    \tau_{\cN'}(z') = \tau_{\cM_1}(J_{\cM}z'J_{\cM}),\quad z'\in\cN'.
\end{equation}
Both traces are independent of the choice of the basis. 
Explicitly, we have 
\begin{align}\label{eqn:: formula for tau_M1}
    \tau_{\cM_1}(xe_{\cN}y) = \delta^{-2}\tau_{\cM}(xy),\quad  x,y\in\cM.
\end{align}

In general, we won't have $\tau_{\cM_1}|_{\cM} = \tau_{\cM}$. 
Let $h_{\cM_1,\cM}$ be the unique positive operator in the center $Z(\cM)$ of $\cM$ such that
\begin{equation}\label{eqn:: the derivative between tau_M1 and tau_M}
    \tau_{\cM}(h_{\cM_1,\cM}x) = \tau_{\cM_1}(x),\quad x\in\cM.
\end{equation} 
We note that by Equation \eqref{eqn:: formula for tau_M1}: 
\begin{equation}\label{eqn:: expectation on e_N in M}
    E^{\cM_1}_{\cM}(e_{\cN}) =  \delta^{-2}h^{-1}_{\cM_1,\cM}.
\end{equation} 

Now we perform Jones basic construction for the inclusion $\cM\subset\cM_1$, obtaining $\cM_2 = J_{\cM_1}\cM'J_{\cM_1} = \cM_1e_{\cM}\cM_1$, where $e_{\cM}$ is the orthogonal projection on $L^2(\cM_1)$ onto $L^2(\cM)$. 
We have that 
\begin{align*}
    \sum_j L_{\tau_{\cM}}(\delta\eta_je_{\cN})L^*_{\tau_{\cM}}(\delta\eta_je_{\cN}) = 1,
\end{align*}
where $L_{\tau_{\cM}}(\delta\eta_je_{\cN})\in \Hom(L^2(\cM)_{\cM},L^2(\cM_1)_{\cM})$ is defined as $L_{\tau_{\cM}}(\delta\eta_je_{\cN})(x\Omega_{\cM}) = \delta\eta_je_{\cN}x\Omega_{\cM_1}$ for all $x\in \cM$. 
Therefore we can define the canonical trace $\tau_{\cM_2}$ as 
\begin{equation}\label{eqn::canonical trace on M2}
    \tau_{\cM_2}(z) = \sum_j \langle z(\eta_je_{\cN}\Omega_{\cM_1}),\eta_je_{\cN}\Omega_{\cM_1}\rangle,\quad z\in \cM_2.
\end{equation}
We check that $\tau_{\cM_2}$ agrees with $\tau_{\cM_1}$ on $\cM_1$. 
In fact, for all $x,y\in\cM$:
\begin{align*}
    \tau_{\cM_2}(xe_{\cN}y) &= \sum_j \langle xe_{\cN}y\eta_je_{\cN}\Omega_{\cM_1},\eta_je_{\cN}\Omega_{\cM_1}\rangle\\
    &= \sum_j\tau_{\cM_1}(xE_{\cN}(y\eta_j)e_{\cN}\eta^*_j)\\
    &= \delta^{-2}\sum_j\tau_{\cM}(xE_{\cN}(y\eta_j)\eta^*_j)\\
    &= \delta^{-2}\tau_{\cM}(xy) = \tau_{\cM_1}(xe_{\cN}y).
\end{align*}
In addition, we have that 
\begin{align*}
    \tau_{\cM_2}(xe_{\cN}ye_{\cM}) &= \delta^{-2}\tau_{\cM_2}(xh^{-1}_{\cM_1,\cM}ye_{\cM})\\
    &= \delta^{-2}\sum_j \langle xyh^{-1}_{\cM_1,\cM}e_{\cM}(\eta_je_{\cN}\Omega_{\cM_1}),\eta_je_{\cN}\Omega_{\cM_1}\rangle\\
    &= \delta^{-4}\sum_j \langle xyh^{-1}_{\cM_1,\cM}\eta_jh^{-1}_{\cM_1,\cM}\Omega_{\cM_1},\eta_je_{\cN}\Omega_{\cM_1}\rangle\\
    &= \delta^{-4}\tau_{\cM_1}(xyh^2_{\cM,\cM_1})\\
    &= \delta^{-2}\tau_{\cM_1}(xe_{\cN}yh^{-1}_{\cM_1,\cM}),
\end{align*}
so we obtain that
\begin{equation}\label{eqn:: expectation of e_M}
    E^{\cM_2}_{\cM_1}(e_{\cM}) = \delta^{-2}h^{-1}_{\cM_1,\cM}.
\end{equation}
\subsection{Bimodules}
The point of view of bimoudle theory is indispensable for our analysis. 
Our reference for basic bimodule theory for finite von Neumann algebras is \cite{Popa2010}.

We view $L^2(\cM)$ together with the left action of $\cM$ and the right action of $\cN$ as an $\cM$-$\cN$ bimodule and denote it as $_{\cM}L^2(\cM)_{\cN}$. 
Then we identify $\cM_1 = \End(L^2(\cM)_{\cN})$, and $\cM'\cap\cM_1 = \End(_{\cM}L^2(\cM)_{\cN})$. 
Similarly with the left action of $\cN$ and the right action of $\cM$, $\End(_{\cN}L^2(\cM)_{\cM})$ is identified with $\cN'\cap\cM$. 
Notice that we have an anti-isomorphism between $\cN'\cap\cM$ and $\cM'\cap\cM_1$ given by modular conjugation on $L^2(\cM)$. 

Now we consider the bimodule $_{\cM}L^2(\cM_1)_{\cM}$. 
It well known \cite{Bisch1997} that $_{\cM}L^2(\cM_1)_{\cM}$ is unitarily equivalent to $_{\cM}L^2(\cM)\otimes_{\cN} L^2(\cM)_{\cM}$, with the equivalence given by:
\begin{align*}
    \delta xe_{\cN}y\Omega_{\cM_1}\mapsto x\Omega_{\cM}\otimes_{\cN} y\Omega_{\cM},\quad x,y\in \cM.
\end{align*}
Because of this equivalence, we will identify $\cM_2$ as $\End(L^2(\cM)\otimes_{\cN} L^2(\cM)_{\cM})$. 
It follows that $\End(_{\cM}L^2(\cM)\otimes_{\cN} L^2(\cM)_{\cM})$ is identified with $\cM'\cap \cM_2$. 

By the isomorphism above, the standard left action of $\cM_1$ on $L^2(\cM_1)$ translates into the action on $L^2(\cM)\otimes_{\cN} L^2(\cM)$ as 
\begin{align*}
    xe_{\cN} y(x_0\Omega_{\cM}\otimes_{\cN}y_0\Omega_{\cM}) = xE_{\cN}(yx_0)\Omega_{\cM}\otimes_{\cN}y_0\Omega_{\cM},\quad x,x_0,y,y_0\in \cM.
\end{align*}
That is, $\cM_1$ acts on the first component. 
Thus the inclusion $\cM_1\hookrightarrow \cM_2$ corresponds to the inclusion
\begin{align*}
    \End(L^2(\cM)_{\cN})\ni z\mapsto z\otimes_{\cN}1\in \End(L^2(\cM)\otimes_{\cN} L^2(\cM)_{\cM}).
\end{align*}
The coincidence of $\tau_{\cM_2}$ and $\tau_{\cM_1}$ on $\cM_1$ can then be expressed as 
\begin{equation}\label{eqn:: trace on M2 agrees wit trace on M1}
    \tau_{\cM_2}(z\otimes_{\cN}1) = \tau_{\cM_1}(z),\quad z\in \cM_1.
\end{equation}
\subsection{Completely positive maps}
Suppose $\mathcal{A}$ and $\mathcal{B}$ are von Neumann algebras and  $\Phi:\mathcal{A}\rightarrow \mathcal{B}$ is a linear map.
The map $\Phi$ is called \textbf{positive} if $\Phi(\mathcal{A}_+)\subseteq\mathcal{B}_+$.
The map $\Phi$ is \textbf{completely positive} if 
\begin{align*}
&\Phi\otimes \id_n:\mathcal{A}\otimes M_n(\mathbb{C})\rightarrow \mathcal{B}\otimes M_n(\mathbb{C})
\end{align*}
is positive for all positive integer $n$, where $(\Phi\otimes id_n)(x_{ij})^n_{i,j=1}=(\Phi(x_{ij}))^n_{i,j=1}$, $x_{ij}\in\mathcal{A}$. The map $\Phi$ is called \textbf{unital} if $\Phi(1_{\mathcal{A}})=1_{\mathcal{B}}$; \textbf{faithful} if $\Phi(x^*x)\neq 0$ whenever $x\neq 0$. 
The map $\Phi$ is called \textbf{normal} if it is continuous with respect to the ultraweak topology of $\mathcal{A}$ and $\mathcal{B}$.
By a \textbf{quantum channel} we mean a normal unital completely positive map, and we denote the set of quantum channels from $\mathcal{A}$ to $\mathcal{B}$ as $\Ch(\mathcal{A},\mathcal{B})$. 
Note that $\Ch(\mathcal{A},\mathcal{B})$ is a convex set.

% \begin{notation}
% Suppose $\Phi:\mathcal{A}\to \mathcal{B}$ is a completely positive map and $0\leq h\in \mathcal{A}$, $0\leq b\in \mathcal{B}$.
% We denote the completely positive map $\Phi(h^{1/2}\cdot h^{1/2})$ by $\Phi^{h}$ and the completely positive map $b^{1/2}\Phi(\cdot) b^{1/2}$ by $^b\Phi$ if there is no confusion.
% \end{notation}

Now we briefly recall the concept of correspondence (bimodule) of a completely positive map which was introduced by Connes in \cite{Connes1994} and developed further in the context of II$_1$ factors in \cite{Popa1986}. 
Let $\phi$ be a normal faithful state on $\mathcal{B}$, and $(L^2(\mathcal{B},\phi),\pi_{\phi},\Omega_{\phi})$ be the GNS-construction. 
The sesquilinear form $\left\langle \cdot, \cdot \right\rangle_0$ on $\mathcal{A}\otimes L^2(\mathcal{B}, \phi)$ is defined as
\begin{align*}
\left\langle a_1\otimes \xi_1, a_2\otimes \xi_2\right\rangle_0=\left\langle \pi_{\phi}(\Phi(a_2^*a_1))\xi_1, \xi_2\right\rangle,
\end{align*}
whenever $a_1, a_2\in \mathcal{A}$ and $\xi_1,\xi_2\in L^2(\mathcal{B}, \phi)$.
The kernel $\mathcal{K}_{\Phi}$ of the sesquilinear form $\left\langle \cdot, \cdot \right\rangle_0$ is given by
$$\mathcal{K}_{\Phi}=\{x\in\mathcal{A}\otimes L^2(\mathcal{B},\phi): \left\langle x,y\right\rangle_0=0,\forall y\in \mathcal{A}\otimes L^2(\mathcal{B},\phi)\},$$
which is invariant under the left action of $\mathcal{A}$ and the right action of $\mathcal{B}$.
The sesquilinear form $\left\langle \cdot,\cdot \right\rangle_0$ then induces an inner product $\left\langle\cdot,\cdot\right\rangle_{\Phi}$ on the quotient $\mathcal{A}\otimes L^2(\mathcal{B}, \phi)/\mathcal{K}_{\Phi}$, and we denote by $\mathcal{H}^{\Phi}$ the Hilbert space by completing the vector space $\mathcal{A}\otimes L^2(\mathcal{B}, \phi)/\mathcal{K}_{\Phi}$ with respect to $\left\langle\cdot,\cdot\right\rangle_{\Phi}$. 
We denote the quotient map from $\mathcal{A}\otimes L^2(\mathcal{B},\phi)$ to $\mathcal{H}^{\Phi}$ as $[\cdot]_{\Phi}$.
The left action of $\mathcal{A}$ on $\mathcal{H}^{\Phi}$ is denoted as $\pi_{\Phi}$, i.e.
\begin{align*}
\pi_{\Phi}(a)[a_0\otimes\xi]_{\Phi}=[aa_0\otimes \xi]_{\Phi},
\end{align*}
where $a, a_0\in \mathcal{A}$, $\xi\in L^2(\mathcal{B}, \phi)$.
Similarly, the right action of $\mathcal{B}$ on $\mathcal{H}^{\Phi}$ is denoted as $\pi'_{\Phi}$, i.e.
\begin{align*}
\pi'_{\Phi}(b)[a_0\otimes\xi]_{\Phi}=[a_0\otimes \xi b]_{\Phi},
\end{align*}
where $a_0\in\mathcal{A}$ and $b\in \mathcal{B}$.
The intertwiner of right $\mathcal{B}$-modules $v_{\Phi,\phi}:L^2(\mathcal{B},\phi)\rightarrow  \mathcal{H}^{\Phi}$ is defined by 
\begin{align*}
v_{\Phi,\phi}\Omega_{\phi}=[1_{\mathcal{A}}\otimes \Omega_{\phi}]_{\Phi}.
\end{align*}
We have $v_{\Phi}$ is an isometry if and only if $\Phi$ is unital. 
Moreover,
\begin{align}\label{dilation}
    \Phi(\cdot)=v^*_{\Phi,\phi}\pi_{\Phi}(\cdot)v_{\Phi,\phi}.
\end{align}
The triple $(\mathcal{H}^{\Phi},\pi_{\Phi},v_{\Phi,\phi})$ is called a dilation of $\Phi$.

It is clear that $\mathcal{H}^\Phi$ equipped with the actions of $\mathcal{A}$ and $\mathcal{B}$ is an $\mathcal{A}$-$\mathcal{B}$-bimodule, denoted by $_{\mathcal{A}}\mathcal{H}^\Phi_{\mathcal{B}}$. 
It can be shown that the unitary equivalence class of $_{\mathcal{A}}\mathcal{H}^\Phi_{\mathcal{B}}$ is independent of the $\phi$. 
We denote $[1\otimes \Omega_{\phi}]_{\Phi}$ as $\Omega_{\Phi,\phi}$. 
Then $\Omega_{\Phi,\phi}$ is separating for $\End(_{\mathcal{A}}\mathcal{H}^\Phi_{\mathcal{B}})$. 
We remark that $\Omega_{\Phi,\phi}$ does depend on the choice of $\phi$.

\subsection{Completely positive bimodule maps}
Suppose $\mathcal{A}\subset \mathcal{B}$ is an inclusion of von Neumann algebras and $\Phi$ is a normal completely positive map from $\mathcal{B}$ to $\mathcal{B}$.
We say $\Phi$ is a $\mathcal{A}$-$\mathcal{A}$-bimodule map if 
\begin{align*}
\Phi(a_1ba_2)=a_1 \Phi(b) a_2,\quad  a_1,a_2\in\mathcal{A}\text{, and }b\in\mathcal{B}.
\end{align*}
We denoted by $\mathbf{CP}_{\mathcal{A}}(\mathcal{B})$ ($\mathbf{CB}_{\mathcal{A}}(\mathcal{B})$) the set of all completely positive (bounded) $\mathcal{A}$-bimodule maps from $\mathcal{B}$ to $\mathcal{B}$.

Let $\cN\subset\cM$ be a finite inclusion of finite von Neumann algebras.  
Then $E_{\cN}$ is a completely positive $\cN$-bimodule map. 
Moreover, the bimodule $\mathcal{H}^{E_{\cN}}$ can be naturally identified with $_{\cM}L^2(\cM)\otimes_{\cN}L^2(\cM)_{\cM}$. 
Here we recall the following proposition from \cite{Popa1986}: 
\begin{proposition}\label{prop::bimodule of a condiitonal expectation}
    The map $\iota: \mathcal{H}^{E_{\cN}}\to L^2(\cM)\otimes_{\cN}L^2(\cM)$, 
    \begin{align*}
        \iota (x\Omega_{E_{\cN}}y)=x\Omega_{\cM}\otimes_{\cN}y\Omega_{\cM},\quad x,y\in \cM,
    \end{align*}
    extends to an isometric $\cM$-$\cM$ bimodule isomorphism. 
    Moreover $\iota v_{\cN}(x\Omega)=\Omega\otimes_{\cN}x\Omega$ for all $x\in\cM$.
\end{proposition}
\begin{proof}
    For $x_1,y_1,x_2,y_2\in \cM$, 
    \begin{align*}
        \left\langle x_1\Omega_{E_{\mathcal{N}}}y_1, x_2\Omega_{E_{\mathcal{N}}}y_2\right\rangle 
         &=\tau_{\mathcal{M}}(y^*_2E_{\mathcal{N}}(x^*_2x_1)y_2)\\
         &=\left\langle x_1\Omega\otimes_{\mathcal{N}}y_2\Omega, x_2\Omega\otimes_{\mathcal{N}}y_2\Omega \right\rangle\\
         &=\left\langle\iota(x_1\Omega_{E_{\mathcal{N}}}y_1),\iota(x_2\Omega_{E_{\mathcal{N}}}y_2),\right\rangle.
    \end{align*}
    It follows that $\iota$ is well-defined and extends linearly to an surjective isometry. 
    By definition we have $\iota(x\Omega_{E_{\cN}}y)=x\iota(\Omega_{E_{\cN}})y$. 
    Since $\Omega_{E_{\cN}}$ is cyclic for $\cM$-$\cM$ action, $\iota$ is a bimodule isomorphism. 
    Finally for each $x\in\cM$, $\iota v_{\cN}(x\Omega)=\iota(\Omega_{E_{\cN}}x)=\Omega\otimes_{\cN}x\Omega$ as claimed.
\end{proof}
\begin{remark}
    When concerning the trace-preserving conditional expectation $E_{\cN}$ for a finite inclusion $\cN\subset\cM$ of finite von Neumann algebras, we will identify $_{\cM}\mathcal{H}^{E_{\cN}}_{\cM}$ and $_{\cM}L^2(\cM)\otimes_{\cN}L^2(\cM)_{\cM}$ via the map $\iota$ and shorten our notation of $v_{E_{\cN}}$ to $v_{\cN}$. 
    By Section 2.3 we will also identify the above two with $_{\cM}L^2(\cM_1)_{\cM}$. 
\end{remark}

Now we discuss the meaning of Proposition \ref{prop::bimodule of a condiitonal expectation} when $\cN\subset\cM$ is a II$_1$ subfactor with finite index. 
Let $\{\mathscr{P}_{n,\pm}\}_{n\geq 0}$ be its planar algbera. 
Given $\Phi\in \mathbf{CB}_{\cN}(\cM)$, then it defines a $\cN$-$\cN$ bimodule map $V_{\Phi,\tau_{\cM}}$ on $L^2(\cM)$ as the closure of the densely defined map $x\Omega_{\cM}\mapsto \Phi(x)\Omega_{\cM}$. 
We represent $V_{\Phi,\tau_{\cM}}$ by the two-box
$\vcenter{\hbox{
      \begin{tikzpicture}
        \begin{scope}[scale=0.5]% 2-box
        \FillGrey (0.3,-0.5) -- (1,-0.5)
        -- (1,1.5) -- (0.3,1.5);
        \VeryThinLine (0.3,-0.5) -- (0.3,1.5);
        \VeryThinLine (1,-0.5) -- (1,1.5);
        \draw[fill=white] (0,0) rectangle (1.3,1);
        \node at (0.65,0.5) {$\Phi$};
        \end{scope}
      \end{tikzpicture}
    }}\in \mathscr{P}_{2,+}$. According to \cite{HJLW23}, see also Theorem \ref{thm:: Fourier multiplier is derivative}, its inverse Fourier transform
    \begin{align}\label{pciture::Fourier multiplier}
        \vcenter{\hbox{\begin{tikzpicture}
            \begin{scope}[scale=0.7]% 2-box
            \FillGrey (-0.3,-0.5) -- (0.3,-0.5)
            -- (0.3,1.5) -- (-0.3,1.5);
            \FillGrey (1,-0.5) -- (1.6,-0.5)
            -- (1.6,1.5) -- (1,1.5);
            \VeryThinLine (0.3,-0.5) -- (0.3,1.5);
            \VeryThinLine (1,-0.5) -- (1,1.5);
            \draw[fill=white] (0,0) rectangle (1.3,1);
            \node[scale = 0.65] at (0.65,0.5) {$\mathcal{F}^{-1}(\Phi)$};
            \end{scope}
          \end{tikzpicture}}}
        :=\vcenter{\hbox{\begin{tikzpicture}[scale=0.7]
            \begin{scope}% Fourier transform of 2-box
                \FillGrey (-0.6,-0.5) -- (1.9,-0.5)
            -- (1.9,1.5) -- (-0.6,1.5);
            \ThinLine (0.3,-0.5)--(0.3, 1) ..
            controls +(0,0.4) and +(0,0.4) .. (-0.3,1)
            -- (-0.3,-0.5);
            \FillWhite (0.3,-0.5)--(0.3, 1) ..
            controls +(0,0.4) and +(0,0.4) .. (-0.3,1)
            -- (-0.3,-0.5)--(0.3,-0.5);
            \ThinLine (1,1.5)--(1,0) ..
            controls +(0,-0.4) and +(0,-0.4) .. (1.6,0)
            -- (1.6,1.5);
            \FillWhite (1,1.5)--(1,0) ..
            controls +(0,-0.4) and +(0,-0.4) .. (1.6,0)
            -- (1.6,1.5)-- (1,1.5);
            \draw[fill=white] (0,0) rectangle (1.3,1);
            \node at (0.65,0.5) {$\Phi$};
            \end{scope}
            %\HelpLines
        \end{tikzpicture}}}
    \end{align} 
    is the unique operator in $\mathscr{P}_{2,-} = \End(_{\cM}\cM\otimes_{\cN}\cM_{\cM})$ such that
    \begin{align}\label{picture::action of two-box on M}
        \Phi(x)=\vcenter{\hbox{\begin{tikzpicture}
                \begin{scope}[scale=0.6]% Convolution
            \FillWhite (-1.3,1.5) -- (-1.3, -0.5) -- (1.3, -0.5) -- (1.3, 1.5) -- (-1.3,1.5);
            \ThinLine (1.3, -0.5) -- (1.3, 1.5);
            \ThinLine (0.6,1) .. controls +(0,0.4) and +(0,0.4) .. (-0.6,1);
            \ThinLine (0.6,0) .. controls +(0,-0.4) and +(0,-0.4) .. (-0.6,0);
            \draw[line width=2pt] (-1.3,-0.5) -- (-1.3,1.5);
            \FillGrey (0.6,1) .. controls +(0,0.4) and +(0,0.4) .. (-0.6,1)
            -- (-0.6,0) -- (-0.6,0) .. controls +(0,-0.4) and +(0,-0.4) .. (0.6,0)
            -- (0.6,1);
            \FillGrey (1.3,-0.5) -- (1.9,-0.5) -- (1.9,1.5) -- (1.3,1.5);
            \draw[fill=white] (0.3,0) rectangle (1.6,1);
            \node[scale=0.55] at (0.95,0.5) {$\mathcal{F}^{-1}(\Phi)$};
            \draw[fill=white, rounded corners] (-1.6,0) rectangle (-0.3,1);
            \node at (-0.95,0.5) {$x$};
        \end{scope}
        \end{tikzpicture}}}, \quad \forall x\in \mathcal{M}.
    \end{align}
    Note that we have used boxes with rounded corners to indicate elements in $\cM$, and strictly speaking the diagram can NOT be understood in $\mathscr{P}^{\cN\subset\cM}$. 
    % This extended planar algebra which also incorporates $\cM$ will be studied in future group works. 
    The operator $\widehat{\Phi}:=\mathcal{F}^{-1}(\Phi)$ is called the Fourier multiplier of $\Phi$. 
    Equivalently we have, see for instance \cite{HJJLW23}, 
    \begin{align}\label{eqn:: multipler defined using a basis}
        \widehat{\Phi}\left(x\Omega_{\cM}\otimes_{\cN}y\Omega_{\cM}\right)=\delta^{-1}\sum^n_{j=1}x\eta_j\Omega_{\cM}\otimes_{\cN}\Phi(\eta^*_j)y\Omega_{\cM},\quad x,y\in\cM.
    \end{align}
    When $\widehat{\Phi}$ is positive, the two-box $\Phi$ is called $\mathfrak{F}$-positive in \cite{HJLW23} (c.f. Definition 2.6).
    It is known that $\widehat{\Phi}$ is positive iff $\Phi$ is completely positive, 
    and the set $\{\mathcal{F}(\Phi)| \Phi\in \mathbf{CP}_{\cN}(\cM)$ coincides with the positive cone of $\mathscr{P}_{2,-}$. 
    By Proposition \ref{prop::bimodule of a condiitonal expectation}, $\End(_{\cM}\mathcal{H}^{E_{\cN}}_{\cM})$ is naturally isomorphic to $\mathscr{P}_{2,-}$. 
    % Therefore the above result says that completely positive $\cN$-bimodule maps on $\cM$ are in 1-1 correspondence with positive operators in $\End(_{\cM}\mathcal{H}^{E_{\cN}}_{\cM})$. 
    % \begin{remark}
    %     In general, replacing $\mathcal{F}(\Phi)$ in Equation \eqref{picture::action of two-box on M} with operators in $\mathscr{P}_{2,-}$ produces completely bounded $\cN$-bimodule maps.  
    % \end{remark}
    
\section{Pimsner-Popa Entropy for Completely Positive Maps}

In this section, we revisit the Pimsner-Popa entropy for subfactors and broaden extend it to a relative entropy applicable to completely positive bimodule maps.

Suppose $\mathcal{R}$ is finite von Neumann algebra with a normal faithful trace $\tau$, and let $\cM, \cN$ be von Neumann subalgebras of $\mathcal{R}$. 
The Connes-St\o rmer relative entropy \cite{CS1975} is defined as follows:
    \begin{align}\label{def::Entropy for a subfactor}
        H(\cM|\cN)=\sup_{\mathbf{x}\in \mathbf{S}}\sum_i\tau\left(\eta(E_{\cN}(x_i))\right)-\tau \left(\eta(E_{\cM}(x_i))\right),
    \end{align}
    where $\eta(t)=-t\log t$ defined for $t\geq 0$, $E_{\cN}$, $E_{\cM}$ are trace-preserving conditional expectations from $\mathcal{R}$ onto $\cN, \cM$ and the set $\mathbf{S}$ consists of all finite partitions of unity in $\mathcal{R}$, which are finite subsets $\mathbf{x}=\{x_i\}_i \subset \mathcal{R}_+$ such that $\displaystyle \sum_ix_i=1$. 
When $\cM=\mathcal{R}$, and $\cN\subset \cM$ is a II$_1$ subfactor of finite index, Popa and Pimsner \cite{PP1986} showed a nice formula as follows:
    \begin{equation}\label{eqn::formula for H(M|N)}
        H(\cM|\cN)=2\log\delta-\sum_k\tau_{\cM}(f_k)\log\left(\frac{\tau_{\cM}(f_k)}{\tau_{\cN'}(f_k)}\right),
    \end{equation}
    where $\{f_k\}_k$ is a set of atoms in $\cN'\cap \cM$ such that $\displaystyle \sum_k f_k=1$.

Now we propose a generalization of $H(\cM|\cN)$ based on the following observation. 
Using the bimodule property of $E_{\cN}$, we can rewrite the summands in Definition \ref{def::Entropy for a subfactor} as
\begin{align*}
    &\tau\left(\eta(E_{\cN}(x_i))\right)-\tau \left(\eta(E_{\cM}(x_i))\right)\\
    &=-\tau_{\cM}(E_{\cN}(x_i)\log E_{\cN}(x_i))+ \tau_{\cM}(x_i\log x_i)\\
        &=\tau_{\cM}(x_i\log x_i)-\tau_{\cM}(x_i\log E_{\cN}(x_i)).
\end{align*}
In a von Neumman algebra $\mathcal{R}$ with faithful normal tracial state $\tau$, for two positive operators $\rho,\sigma$ the quantity $\tau(\rho\log \rho - \rho\log \sigma)=D_{\tau}(\rho\|\sigma)$ is called the relative entropy between $\rho$ and $\sigma$. 
Notice that in order for $D_{\tau}(\rho\|\sigma)<\infty$, we require $p_{\rho}\leq p_{\sigma}$, where $p_{\rho}$ and $p_{\sigma}$ are range projections of $\rho$ and $\sigma$. We have thus proved the following lemma:
\begin{lemma}\label{lemma::4}
Suppose $\cN\subset \cM$ is an finite inclusion of finite von Neumann algebras. Then 
    \begin{align}\label{expr::2}
        H(\cM|\cN)=\sup_{\mathbf{x}\in \mathbf{S}}\sum_i D_{\tau_{\cM}}(x_i\|E_{\cN}(x_i)),
    \end{align}
where $\mathbf{S}$ is the set of all finite partitions of unity in $\cM$.
\end{lemma}
Having rewritten $H(\cM|\cN)$ in the above form, it is now natural to replace the identity $id_{\cM}$ and the conditional expectation $E_{\cN}$ by arbitrary completely positive bimoudle maps.

In the rest of this section, we assume $\cN\subset\cM$ is an inclusion of finite von Neumann algebras and we fix a normal faithful trace $\tau_{\cM}$ on $\cM$. 
\begin{definition}\label{def::CSEntropy}
    Suppose $\cN\subseteq \cM$ and $\Phi, \Psi\in \mathbf{CP}_{\cN}(\cM)$. 
    Define the Connes-St{\o}rmer entropy between $\Phi$ and $\Psi$ as
    \begin{equation}\label{eqn:: Pimsner-Popa relative entropy between cp maps}
        H(\Phi|\Psi)=\sup_{\mathbf{x}\in \mathbf{S}}\sum_i D_{\tau_{\cM}}(\Phi(x_i)\|\Psi(x_i)),
    \end{equation}
    where $\mathbf{S}$ is the set of all finite partitions of unity in $\cM$.
\end{definition}
\begin{remark}
    Suppose $\mathbb{C}\subset P,Q\subset\cM$ are von Neumann subalgebras, then $H(E_{P}|E_{Q})$ equals to the original definition of Connes and St{\o}rmer only if $Q\subset P$. 
\end{remark}
Notice that for the right hand side of Equation \eqref{expr::2} to be finite, we need $p_{E_{\cN}(x)}\geq p_x$ for every $x\in \cM_+$, the set of all positive elements in $\cM$.
In the case of a finite index subfactor, this assumption is full-filled by the Pimsner-Popa inequality \cite{PP1986} which asserts that 
\begin{align}\label{Pimsner-Popa inequality}
    E_{\cN}(x)\geq \delta^{-2}x,\quad \forall x\in \cM_+.
\end{align}
To make sure that $H(\Phi|\Psi)$ is well-defined, we adopt the following definition of majorization between completely positive maps. 
This notion has already appeared in \cite{Connes1994} and \cite{Popa2010}.
\begin{definition}\label{def::1}
    Suppose $\Psi,\Phi:\mathcal{A}\to\mathcal{B}$ is normal completely positive.
    We say $\Phi$ is majorized by $\Psi$ and write $\Phi \preccurlyeq\Psi$ if there is a positive scalar $c$ such that $c\cdot \Psi-\Phi$ is completely positive. 
    We write $\Phi\sim\Psi$ if both $\Phi \preccurlyeq\Psi$ and $\Psi \preccurlyeq\Phi$ hold.
\end{definition}
\begin{remark}\label{rmk:: E_N dominates bimodule maps}
    In \cite{HJJLW23} it has been proved that for a finite inclusion of finite von Neumann algebras $\cN\subset\cM$, $\Phi\in\mathbf{CP}_{\cN}(\cM)$ implies that $\Phi\preccurlyeq E_{\cN}$. 
    Let $\{\eta_i\}^m_{i=1}$ be a basis of $L^2(\cM)_{\cN}$ and define the $m\times m$ positive $\cM$-valued matrix as $[G]_{ij} = \eta^*_i \eta_j$. 
    It is also proved that for $\Phi,\Psi\in \mathbf{CP}_{\cN}(\cM)$, $\Phi\preccurlyeq \Psi$ is equivalent to $\mathbf{supp} \Phi(G)\leq \mathbf{supp}  \Psi(G)$ as positive operators. 
\end{remark}
Our goal in this section is to derive an upper bound for $H(\Phi|\Psi)$ for completely positive $\cN$-bimodule maps in terms of their Fourier multipliers. 

%%% Construct the Fourier multiplier for completely positive bimodule maps.
We assume that $\cN\subset\cM$ is of finite index. 
In this and the next section, we denote by $\Omega$ the cyclic separating tracial vector in $L^2(\cM)$. 
Denote the vector state on $\End(_{\cM}L^2(\cM)\otimes_{\cN}L^2(\cM)_{\cM})$ implemented by $\Omega\otimes_{\cN}\Omega$ as $\omega_{\cN}$. 
Since $\Omega\otimes_{\cN}\Omega$ generates $_{\cM}L^2(\cM)\otimes_{\cN}L^2(\cM)_{\cM}$ under $\cM$-$\cM$ action, the state $\omega_{\cN}$ is faithful. 
As we shall see in Lemma \ref{lemma::defining Delta}, $\omega_{\cN}$ may not be a trace. 
In fact, if $\cN\subset\cM$ is a II$_1$ subfactor of finite index, then $\omega_{\cN}$ is a trace if and only if $\cN\subset \cM$ is extremal \cite{Burns2011}. 
\begin{remark}
    As pointed out by Longo in \cite{Longo1989} (see the Remark after Theorem 5.5), the extremality of a subfactor corresponds to the minimality of the trace preserving conditional expectation. 
\end{remark}
\begin{proposition}\label{prop:: characterize the state omega}
    For any $\Phi\in\mathbf{CP}_{\cN}(\cM)$ the following hold:\\
    \begin{enumerate}
        \item for all $x\in\cM$ we have $\Phi(x)=\delta v^*_{\cN} (x\otimes_{\cN}1) \widehat{\Phi}v_{\cN}$;\\
        \item $\delta \omega_{\cN}(\widehat{\Phi})=\tau_{\cM}(\Phi(1))$.
    \end{enumerate}
\end{proposition}
\begin{proof}
    Recall that by Proposition \ref{prop::bimodule of a condiitonal expectation} we identify $\mathcal{H}^{E_{\cN}}$ with $L^2(\cM)\otimes_{\cN}L^2(\cM)$ as $\cM$-$\cM$ bimodules, and thus $v_{\cN}:L^2(\cM)_{\cM}\to L^2(\cM)\otimes_{\cN}L^2(\cM)_{\cM}$ is defined by 
    \begin{align*}
        v_{\cN}(y\Omega)=\Omega\otimes_{\cN}y\Omega,\quad\forall y\in\cM.
    \end{align*}
    Therefore for any $x,y_1,y_2\in\cM$, by Equation \eqref{eqn:: multipler defined using a basis}
    \begin{align*}
        \left\langle v^*_{\cN} x \widehat{\Phi}v_{\cN}y_1\Omega,y_2\Omega\right\rangle 
        &= \left\langle \widehat{\Phi}\left(x\Omega\otimes y_1\Omega\right),\Omega\otimes y_2\Omega \right\rangle\\
        &= \delta^{-1}\tau_{\cM} (y^*_2\Phi(x)y_1)\\
        &= \delta^{-1}\left\langle \Phi(x)y_1\Omega,y_2\Omega\right\rangle.
    \end{align*}
    This proves (1). 
    Since
     \begin{align*}
        \omega_{\cN}(\widehat{\Phi})=\left\langle \widehat{\Phi}(\Omega\otimes_{\cN}\Omega),\Omega\otimes_{\cN}\Omega\right\rangle=\delta^{-1}\tau_{\cM}(\Phi(1)),
     \end{align*}
     (2) holds.
\end{proof}
\begin{corollary}\label{corollary:: comparison of Fourier multiplier}
    For any $\Phi,\Psi\in \mathbf{CP}_{\cN}(\cM)$, we have $\Phi\preccurlyeq \Psi$ if and only if $\mathbf{supp}\widehat{\Phi}\leq \mathbf{supp}\widehat{\Psi}$.
\end{corollary}
\begin{proof}
    For $c>0$, the Fourier multiplier of $c\Psi-\Phi$ is easily seen to be $c\widehat{\Psi}-\widehat{\Phi}$, hence 
    $c\Psi-\Phi$ being completely positive implies $c\widehat{\Psi}-\widehat{\Phi}\geq 0$. 
    Conversely suppose $c\widehat{\Psi}-\widehat{\Phi}\geq 0$, and let $R$ be its positive square root, then for all $x\in \cM$, 
    \begin{align*}
        c\Psi(x)-\Phi(x) = \delta v^*_{\cN}R(x\otimes_{\cN}1)R v_{\cN},
    \end{align*}
    which proves $c\Psi-\Phi$ is completely positive.
\end{proof}
We now find the density operator of the state $\omega_{\cN}$ with respect to $\tau_{\cM_2}$. 
Notice that $\tau_{\cM_2}(z\otimes_{\cN} 1)=\tau_{\cM_1}(z)$ for all $z\in \End(_{\cM}L^2(\cM)_{\cN})$. 
Define an isometry $u_{\cN}\in\Hom(_{\cM}L^2(\cM)_{\cN}, _{\cM}L^2(\cM)\otimes_{\cN}L^2(\cM)_{\cN})$ as 
\begin{align}\label{def:: u_N}
    u_{\cN}(\xi)=\xi\otimes_{\cN}\Omega,\quad \xi\in L^2(\cM),
\end{align}
so that $u^*_{\cN}(\xi\otimes_{\cN}y\Omega) = \xi\cdot E_{\cN}(y)$. 
Therefore for any $A\in \End(_{\cM}L^2(\cM)\otimes_{\cN}L^2(\cM)_{\cN})$, $u^*_{\cN}Au_{\cN}\in \End(_{\cM}L^2(\cM)_{\cN})$. 
It now follows from Equation \eqref{def:: u_N} and the definition of $\omega_{\cN}$ that
\begin{align*}
    \omega_{\cN}(A) &= \tau_{\cM}(J_{\cM}u^*_{\cN}A^*u_{\cN}J_{\cM}).
\end{align*}
In addition, by Equation \eqref{eqn::canonical trace on M1} and Equation \eqref{eqn::canonical trace on M2} we have
\begin{align*}
    \tau_{\cM_2}(A) &= \tau_{\cM_1}(u^*_{\cN}Au_{\cN}).
\end{align*}
\begin{lemma}\label{lemma::defining Delta}
    There exists a unique operator $\Delta>0$ in $\End(_{\cM}L^2(\cM)_{\cN})$ such that 
    \begin{align}\label{defining property of Delta}
       \tau_{\cM_2}\left((\Delta\otimes_{\cN} 1)A\right) =  \omega_{\cN}(A),\quad \forall A\in \End(_{\cM}L^2(\cM)\otimes_{\cN}L^2(\cM)_{\cN}).
    \end{align}
\end{lemma}
\begin{proof}
Let $\Delta$ be the unique positive operator in $\cM'\cap\cM_1$(which we identify as $\End(_{\cM}L^2(\cM)_{\cN})$) such that $\tau_{\cM_1}(\Delta z)=\tau_{\cM}(J_{\cM}z^*J_{\cM})$ for all $z\in\cM'\cap \cM_1$. 

Observe that operators in $\cM'\cap \cM_1$ commute with the right action of $\cN$ on $L^2(\cM)$. 
Thus for $x,y\in\cM$:
\begin{align*}
    u^*_{\cN}(\Delta\otimes_{\cN} 1)(x\Omega\otimes_{\cN}y\Omega) &= u^*_{\cN}\left( \Delta(x\Omega)\otimes_{\cN}y\Omega\right) =\Delta(x\Omega)\cdot E_{\cN}(y)\\
    &= \Delta(xE_{\cN}(y)\Omega) = \Delta u^*_{\cN}(x\Omega\otimes_{\cN}y\Omega),
\end{align*}
and we get $u^*_{\cN}\Delta\otimes_{\cN} 1=\Delta u^*_{\cN}$. 
This implies that for all $A\in \End(_{\cM}L^2(\cM)\otimes_{\cN}L^2(\cM)_{\cN})$,
\begin{align*}
    \tau_{\cM_2}\left((\Delta\otimes_{\cN} 1)A\right) &= \tau_{\cM_1}(\Delta u^*_{\cN}Au_{\cN}) = \tau_{\cM}(J_{\cM}u^*_{\cN}A^*u_{\cN}J_{\cM}) = \omega_{\cN}(A).
\end{align*}

To prove uniqueness, suppose $\Delta' \in \End(_{\cM}L^2(\cM)_{\cN})$ satisfies Equation \eqref{defining property of Delta}. 
Then for every $z\in \End(_{\cM}L^2(\cM)_{\cN})$:
\begin{align*}
    \tau_{1}(\Delta' z) &= \tau_{\cM_2} \left(\Delta' z\otimes_{\cN} 1\right) = \omega_{\cN}(z\otimes_{\cN} 1)\\
    &= \tau_{\cM}(J_{\cM}z^*J_{\cM}) = \tau_{1}(\Delta z),
\end{align*}
hence $\Delta'=\Delta$.
\end{proof}
If there is no confusion, we will denote $\Delta\otimes_{\cN}1$ as $\Delta$.
\begin{remark}
    If $\cN\subset\cM$ is a II$_1$ subfactor of finite index, then pictorially we have:
    \begin{align*}
         \omega_{\cN}(\widehat{\Phi}_0)= \delta^{-1}\tau_{\cM}(\Phi_0(1)) = \delta^{-1}\tau_{\cM}\left(
            \vcenter{\hbox{\begin{tikzpicture}
            \begin{scope}[scale=0.55]% Convolution
            \ThinLine (1.3, -0.5) -- (1.3, 1.5);
            \ThinLine (0.6,1) .. controls +(0,0.4) and +(0,0.4) .. (-0.6,1);
            \ThinLine (-0.6,1) -- (-0.6,0);
            \ThinLine (0.6,0) .. controls +(0,-0.4) and +(0,-0.4) .. (-0.6,0);
            \FillGrey (0.6,1) .. controls +(0,0.4) and +(0,0.4) .. (-0.6,1)
            -- (-0.6,0) -- (-0.6,0) .. controls +(0,-0.4) and +(0,-0.4) .. (0.6,0)
            -- (0.6,1);
            \FillGrey (1.3,-0.5) -- (1.9,-0.5) -- (1.9,1.5) -- (1.3,1.5);
            \draw[fill=white] (0.3,0) rectangle (1.6,1);
            \node[scale=1] at (0.95,0.5) {$\widehat{\Phi}_0$};
        \end{scope}
        \end{tikzpicture}}}
        \right) = 
        \delta^{-2}\vcenter{\hbox{\begin{tikzpicture}
            \begin{scope}[scale=0.55]% 2-box
                \FillGrey (0.3,0) .. controls  + (0,-0.4) and + (0,-0.4) .. (-0.3,0)
                -- (-0.3,1) .. controls + (0,0.4) and + (0,0.4).. (0.3,1);
                \FillGrey (1,0) .. controls  + (0,-0.4) and + (0,-0.4) .. (1.6,0)
                -- (1.6,1) .. controls + (0,0.4) and + (0,0.4).. (1,1);
                \VeryThinLine (0.3,0) .. controls  + (0,-0.4) and + (0,-0.4) .. (-0.3,0)
                -- (-0.3,1) .. controls + (0,0.4) and + (0,0.4).. (0.3,1);
                \VeryThinLine (1,0) .. controls  + (0,-0.4) and + (0,-0.4) .. (1.6,0)
                -- (1.6,1) .. controls + (0,0.4) and + (0,0.4).. (1,1);
            \draw[fill=white] (0,0) rectangle (1.3,1);
            \node at (0.65,0.5) {$\widehat{\Phi}_0$};
            \end{scope}
            \end{tikzpicture}}}.
    \end{align*}
    Therefore we obtain a pictorial characterization of $\Delta\otimes_{\cN} 1$:
    \begin{align*}
        \vcenter{\hbox{\begin{tikzpicture}
            \begin{scope}[scale=0.55]
                \FillGrey (0.3,0) .. controls  + (0,-0.4) and + (0,-0.4) .. (-0.3,0)
                -- (-0.3,1) .. controls + (0,0.4) and + (0,0.4).. (0.3,1);
                \FillGrey (1,0) .. controls  + (0,-0.4) and + (0,-0.4) .. (1.6,0)
                -- (1.6,1) .. controls + (0,0.4) and + (0,0.4).. (1,1);
                \VeryThinLine (0.3,0) .. controls  + (0,-0.4) and + (0,-0.4) .. (-0.3,0)
                -- (-0.3,1) .. controls + (0,0.4) and + (0,0.4).. (0.3,1);
                \VeryThinLine (1,0) .. controls  + (0,-0.4) and + (0,-0.4) .. (1.6,0)
                -- (1.6,1) .. controls + (0,0.4) and + (0,0.4).. (1,1);
            \draw[fill=white] (0,0) rectangle (1.3,1);
            \node at (0.65,0.5) {$x$};
            \end{scope}
            \end{tikzpicture}}} = 
        \vcenter{\hbox{\begin{tikzpicture}
            \begin{scope}[scale=0.55]
                \FillGrey (-0.6, -0.9) -- (2.6,-0.9) --(2.6, 2.4) --(-0.3, 2.4) --(-0.3,-0.9);
                \FillWhite (0.3,0) .. controls  + (0,-1) and + (0,-1) .. (2.3,0)
                -- (2.3,1.5) .. controls + (0,1) and + (0,1).. (0.3,1.5) -- (0.3,1)--(0.3,0);
                \FillGrey (1,0) .. controls  + (0,-0.4) and + (0,-0.4) .. (1.6,0)
                -- (1.6,1.5) .. controls + (0,0.4) and + (0,0.4).. (1,1.5)--(1,1)--(1,0);
                \VeryThinLine (0.3,0) .. controls  + (0,-1) and + (0,-1) .. (2.3,0)
                -- (2.3,1.5) .. controls + (0,1) and + (0,1).. (0.3,1.5) -- (0.3,1);
                \VeryThinLine (1,0) .. controls  + (0,-0.4) and + (0,-0.4) .. (1.6,0)
                -- (1.6,1.5) .. controls + (0,0.4) and + (0,0.4).. (1,1.5)--(1,1);
                \draw[fill=white] (0,0) rectangle (1.3,1);
                \node at (0.65,0.5) {$x$};
                \draw[fill=white] (0,1.1) rectangle (0.6,1.7);
                \node [scale=0.8] at (0.3,1.4) {$\Delta$};
            \end{scope}
            \end{tikzpicture}}},\quad \forall x\in \mathscr{P}^{\cN\subset \cM}_{2,+}.
    \end{align*}
    Additionally, it is worth noting that the operator $\Delta$ has been introduced in Burns' thesis \cite{Burns2011}, where it is denoted as $\tilde{w}$ (see notation 2.2.13 and lemma 2.2.14 on page 35).
    % We further remark that the operator $\Delta$ has appeared in Burns' thesis \cite{Burns2011}, in which it is denoted as $\tilde{w}$ (notation 2.2.13 and lemma 2.2.14 on page 35). 
\end{remark}
Next we derive an upper bound for $H(\Phi\vert \Psi)$, and we will find $\Delta$ naturally appear in this process. 
In fact as a direct consequence of Equation \eqref{Pimsner-Popa inequality}, one can show that $H(\cM|\cN)=H(id_{\cM},E_{\cN})\leq \log\delta^2$ when $\cN\subset\cM$ is a II$_1$ subfacotr of finite index (c.f. Proposition 3.5 of \cite{PP1986}). 
However, as formula \eqref{eqn::formula for H(M|N)} indicates, this upper bound is optimal only if $\cN\subset \cM$ is extremal.  
\begin{theorem}\label{thm::popaentropycompletepositive}
    Suppose $\mathcal{N}\subset\mathcal{M}$ is a finite inclusion of finite von Neumann algebras, and
    $\Phi,\Psi\in \mathbf{CP}_{\cN}(\cM)$ with $\Phi\preccurlyeq \Psi$. 
    Then
    \begin{align}\label{eqn::upper bound for H}
        H(\Phi|\Psi)\leq \delta D_{\tau_{\cM_2}}(\Delta^{1/2} \widehat{\Phi}\Delta^{1/2} \|\Delta^{1/2} \widehat{\Psi}\Delta^{1/2} ).
    \end{align}
\end{theorem}
\begin{proof}
    Let $\{x_j\}_{1\leq j\leq n}\subset\cM$ be a finite partition of unity.
    We define a completely positive map $T:\End(_{\cM}L^2(\cM)\otimes_{\cN}L^2(\cM)_{\cM})\to \cM^{\oplus n}$ by
    \begin{align*}
        T(A)=\bigg(v^*_{\cN}x_j\Delta^{-1/2} A\Delta^{-1/2} v_{\cN}\bigg)_{1\leq j\leq n}.
    \end{align*}

    Define the trace $Tr$ on $\cM^{\oplus n}$ as $Tr\big((x_j)_{1\leq j\leq n}\big) = \sum_j\tau_{\cM}(x_j)$.
    We shall show that $T$ is trace-preserving. 
    Then for all $A$ in $\End(_{\cM}L^2(\cM)\otimes_{\cN}L^2(\cM)_{\cM})$
    \begin{align*}
        Tr(T(A))&=\sum^n_{j=1}\tau_{\cM}(v^*_{\cN}x_j\Delta^{-1/2}A\Delta^{-1/2}v_{\cN})\\
        &=\tau_{\cM}(v^*_{\cN}\Delta^{-1/2}A\Delta^{-1/2}v_{\cN})\\
        &=\omega_{\cN}(\Delta^{-1/2}A\Delta^{-1/2})=\tau_{\cM_2}(A),
    \end{align*}
    where the last equality follows from the Lemma \ref{lemma::defining Delta}. 
    Thus we deduce that $Tr\circ T=\tau_{\cM_2}$. 
    Now by the statement (1) of Proposition \ref{prop:: characterize the state omega}:
    \begin{align*}
        \delta D_{Tr}(T(\Delta^{1/2}\widehat{\Phi}\Delta^{1/2})\| T(\Delta^{1/2}\widehat{\Psi}\Delta^{1/2}))=\sum^n_{j=1} D_{\tau_{\cM}}(\Phi(x_j)\|\Psi(x_j)).
    \end{align*}
    Applying the data processing inequality, we have
    \begin{align*}
        \sum^n_{j=1} D_{\tau_{\cM}}(\Phi(x_j)\|\Psi(x_j))&= \delta D_{Tr}(T(\Delta^{1/2}\widehat{\Phi}\Delta^{1/2}_L)\| T(\Delta^{1/2}\widehat{\Psi}\Delta^{1/2}_L))\\
        &\leq \delta D_{\tau_{\cM_2}}(\Delta^{1/2}\widehat{\Phi}\Delta^{1/2}\|\Delta^{1/2}\widehat{\Psi}\Delta^{1/2}).
    \end{align*}
    This implies, by taking supremum over all finite partitions in $\cM$, 
    \begin{align*}
    H(\Phi|\Psi)\leq \delta D_{\tau_{\cM_2}}(\Delta^{1/2}\widehat{\Phi}\Delta^{1/2}\| \Delta^{1/2}\widehat{\Psi}\Delta^{1/2}).
    \end{align*}
    Hence the theorem follows.
\end{proof}
\begin{remark}
    Another natural trace defined on $\cM'\cap \cM_2$ is given by
    \begin{align*}
        \tau' (x) = \tau_{\cN'}(v^*_{\cN}xv_{\cN}), x\in \cM'\cap \cM_2.
    \end{align*}
    The density operator of $\omega_{\cN}$ with respect to $\tau'$ is given by $1\otimes_{\cN} J_{\cM}\Delta J_{\cM}$, so it is possible to express the upper bound using relative in terms of relative entropy with respect to $\tau'$ as well. 
    % So we obtain
    % \begin{align*}
    %     &\delta D_{\tau_{\cM_2}}(\Delta^{1/2}\widehat{\Phi}\Delta^{1/2}| \Delta^{1/2}\widehat{\Psi}\Delta^{1/2})\\
    %     &= 
    %     \delta D_{\tau'}\big((1\otimes_{\cN}J_{\cM}\Delta^{1/2}J_{\cM}) \widehat{\Phi} (1\otimes_{\cN}J_{\cM}\Delta^{1/2}J_{\cM})| (1\otimes_{\cN}J_{\cM}\Delta^{1/2}J_{\cM}) \widehat{\Psi} (1\otimes_{\cN}J_{\cM}\Delta^{1/2}J_{\cM})\big).
    % \end{align*}
\end{remark}
We now compute the upperbound in Inequality \eqref{eqn::upper bound for H} subject to the case where $\Phi = id_{\cM}$ and $\Psi = E_{\cN}$. 
\begin{lemma}\label{lemma:: Fourier multiplier of identity map}
    Let $h_{\cM_1,\cM}\in Z(\cM)$ be as in Equation \eqref{eqn:: the derivative between tau_M1 and tau_M}, then
    \begin{align*}
        \delta^{-1}\widehat{id_{\cM}} = h_{\cM_1,\cM}e_{\cM},
    \end{align*}
    and consequently $E^{\cM_2}_{\cM_1}(\delta^{-1}\widehat{id_{\cM}}) = \delta^{-2}$.
\end{lemma}
\begin{proof}
    By cyclicity of $e_{\cN}\Omega_{\cM_1}$, it is enough to check that $h_{\cM_1,\cM}e_{\cM} (e_{\cN}\Omega_{\cM_1}) = \delta^{-2}\Omega_{\cM_1} = \delta^{-1}\widehat{id_{\cM}}(e_{\cN}\Omega_{\cM_1})$. 
    Equivalently we need to show
    \begin{align*}
        E^{\cM_1}_{\cM} (e_{\cN}) = \delta^{-2}h^{-1}_{\cM_1,\cM},
    \end{align*} 
 which is nothing but Equation \eqref{eqn:: expectation on e_N in M}. 
\end{proof}
\begin{lemma}\label{lemma::the first term is a projection}
    Let $\cN\subset\cM$ be a finite inclusion of finite von Neumann algebras. 
    The element
    \begin{align*}
        \delta^{-1}\Delta^{1/2}\widehat{id_{\cM}}\Delta^{1/2}
    \end{align*}
    is a projection in $\cM'\cap\cM_2$ equivalent to $e_{\cM}$.
\end{lemma}
\begin{proof}
    By Lemma \ref{lemma:: Fourier multiplier of identity map}, we have
    \begin{align*}
    \delta^{-1}\Delta^{1/2}\widehat{id_{\cM}}\Delta^{1/2} = yy^*, \quad 
    y = \Delta^{1/2}h^{1/2}_{\cM_1,\cM}e_{\cM}.
    \end{align*}
    Therefore it suffices to show that $y^*y = E^{\cM_1}_{\cM}(\Delta)h_{\cM_1,\cM}e_{\cM}$ is a projection. 
    In deed, for any $x\in Z(\cM)$:
    \begin{align*}
        \tau_{\cM_1}(x\Delta) = \tau_{\cM}(Jx^*J) = \tau_{\cM}(x) = \tau_{\cM_1}(xh^{-1}_{\cM_1,\cM}), 
    \end{align*}
    which implies $E^{\cM_1}_{\cM}(\Delta) = h^{-1}_{\cM_1,\cM}$. 
    Therefore $y^*y = h^{-1}_{\cM_1,,\cM} h_{\cM_1,\cM} e_{\cM} = e_{\cM}$ is a projection. 
\end{proof}
% \begin{proof}
%     Let $\{\xi_i\}_i$ be a Pimsner-Popa basis of $L^2(\cM)_{\cN}$, then after identifying $1\otimes_{\cN}1$ with $\delta e_{\cN}\Omega_{\cM_1}\in L^2(\cM_1)$:
%     \begin{align*}
%         \widehat{id_{\cM}}(e_{\cN}\Omega_{\cM_1}) = \delta^{-1}\sum_i \xi_i e_{\cN}\xi^*_i\Omega_{\cM_1} = \delta^{-1}\Omega_{\cM_1}.
%     \end{align*}
%     So we have
%     \begin{align*}
%         \widehat{id_{\cM}}\Delta\widehat{id_{\cM}}(e_{\cN}\Omega_{\cM_1}) = \delta^{-2}\sum_{i}\xi_i\Delta^{-1}_0\xi^*_i\Omega_{\cM_1}.
%     \end{align*}
%     Since $\Delta^{-1}_0 \in \cN'\cap\cM$, we have $\delta^{-2}\sum_{i}\xi_i\Delta^{-1}_0\xi^*_i\in Z(\cM)$. 
%     And for any $x\in Z(\cM)$,
%     \begin{align*}
%         \delta^{-2}\sum_{i}\tau_{\cM}(x\xi_i\Delta^{-1}_0\xi^*_i) = \delta^{-2}\sum_{i}\tau_{\cM}(\xi_i x\Delta^{-1}_0\xi^*_i) = \tau_{\cN'}(x\Delta^{-1}_0) = \tau_{\cM}(x).
%     \end{align*}
%     Thus $\delta^{-2}\sum_{i}\xi_i\Delta^{-1}_0\xi^*_i = 1$ since $\tau_{\cM}$ is faithful. 
%     We therefore get:
%     \begin{align*}
%         \widehat{id_{\cM}}\Delta\widehat{id_{\cM}}(e_{\cN}) = 1 = \delta\widehat{id_{\cM}}(e_{\cN}),
%     \end{align*}
%     and
%     \begin{align*}
%         (\delta^{-1}\Delta^{1/2}\widehat{id_{\cM}}\Delta^{1/2})^2 = \delta^{-2} \Delta^{1/2}(\widehat{id_{\cM}}\Delta\widehat{id_{\cM}})\Delta^{1/2} = \delta^{-1} \Delta^{1/2}\widehat{id_{\cM}}\Delta^{1/2},
%     \end{align*}
%     and the result follows.
% \end{proof}
\begin{definition}\label{def::Delta_0}
    Let $\cN\subset\cM$ be a finite inclusion of finite von Neumann algebras.
    We define $\Delta_0$ to be the positive operator $\in \cN'\cap \cM$ such that $\tau_{\cM}(\Delta_0\cdot) = \tau_{\cN'}$ as traces on $\cN'\cap\cM$. 
    Equivalently, we have $\Delta_0 = J_{\cM}\Delta^{-1} J_{\cM}$. 
\end{definition}
\begin{proposition}\label{prop:: explicit upperbound}
    Let $\cN\subset \cM$ and $\Delta_0$ be as above. 
    Then
    \begin{align*}
        \delta D_{\tau_{\cM_2}}(\Delta^{1/2}\widehat{id_{\cM}}\Delta^{1/2}\|\Delta^{1/2}\widehat{E_{\cN}}\Delta^{1/2}) = 2\log \delta+\tau_{\cM}(\log \Delta_0).
    \end{align*}
\end{proposition}
\begin{proof}
    Firstly, we have $\widehat{E_{\cN}} = \delta^{-1}$. 
    So $\Delta^{1/2}\widehat{E_{\cN}}\Delta^{1/2} = \delta^{-1}\Delta$.
    By Lemma \ref{lemma::the first term is a projection},  the first term $e = \delta^{-1}\Delta^{1/2}\widehat{id_{\cM}}\Delta^{1/2}$ is a projection equivalent to $e_{\cM}$, so $\tau_{\cM_2} (e) = \tau_{\cM_2}(e_{\cM}) = \delta^{-2}$. 
    Now we have
    \begin{align*}
        \delta D_{\tau_{\cM_2}}(\Delta^{1/2}\widehat{id_{\cM}}\Delta^{1/2}\| \Delta^{1/2}\widehat{E_{\cN}}\Delta^{1/2}) &= \delta^2 D_{\tau_{\cM_2}}(e\vert \delta^{-2}\Delta)\\
        &= \delta^2 \tau_{\cM_2} (e\log e -e\log \delta^{-2}\Delta)\\
        &= -\delta^2 \tau_{\cM_2} (e\log \delta^{-2}\Delta)\\
        &= 2\log\delta - \delta^2\tau_{\cM_2} (e\log \Delta).
    \end{align*}
   Note that $\Delta\in\cM'\cap\cM_1$. We obtain
    \begin{align*}
        \delta^2\tau_{\cM_2} (e\log \Delta) &=  \delta^2\tau_{\cM_1} (E^{\cM_{2}}_{\cM_1}(e)\log \Delta)\\
        &= \delta^2\tau_{\cM_1} (E^{\cM_{2}}_{\cM_1}(\delta^{-1}\widehat{id_{\cM}})\Delta\log \Delta)\\
        &= \tau_{\cM_1}(\Delta\log \Delta).
    \end{align*}
    By the fact that $\Delta^{-1}_0 = J_{\cM}\Delta J_{\cM}$, we see
    \begin{align*}
\tau_{\cM'}(\log\Delta) = \tau_{\cM'}(J_{\cM}\log \Delta^{-1}_0J_{\cM}) = -\tau_{\cM}(\log\Delta_0).
    \end{align*}
    This completes the proof.
\end{proof}
Let $\{f_k\}_k$ be a set of atoms in $\cN'\cap\cM$. 
Then 
\begin{align*}
    \Delta_0 = \sum_k \frac{\tau_{\cN'}(f_k)}{\tau_{\cM}(f_k)}f_k,
\end{align*}
and 
\begin{align*}
    -\tau_{\cM}(\log\Delta_0) = \sum_k \tau_{\cM}(f_k)\log\frac{\tau_{\cM}(f_k)}{\tau_{\cN'}(f_k)}.
\end{align*}
Therefore we obtain
\begin{align*}
    2\log\delta + \tau_{\cM}(\log\Delta_0) = 2\log\delta - \sum_k \tau_{\cM}(f_k)\log\frac{\tau_{\cM}(f_k)}{\tau_{\cN'}(f_k)}.
\end{align*}
This precisely corresponds to the formula for $H(\cM\vert\cN)$ obtained by Pimsner and Popa when $\cN\subset\cM$ is subfactor. 
In particular, it means that if $\cN\subset \cM$ is a subfactor of finite index, then
\begin{align*}
    H(\cM\vert\cN) = \delta D_{\tau_{\cM_2}}(\Delta^{1/2}\widehat{id_{\cM}}\Delta^{1/2}\| \Delta^{1/2}\widehat{E_{\cN}}\Delta^{1/2}).
\end{align*}
This coincidence suggests us to look for the case of equality in Equation \eqref{eqn::upper bound for H}, 
which we explore in the next section. 

\section{The Downward Jones Basic Construction}
Our goal in the this section is to prove that the equality in Theorem \ref{thm::popaentropycompletepositive} does occur if we assume $\cN\subset\cM$ admits downward Jones basic construction. 

We say the inclusion $\cN\subset \cM$ admits a downward Jones basic construction if there exists a subalgebra $\cN_{-1}\subset\cN$ and a trace-preserving $*$-isomorphism $\alpha:\cM\to J_{\cN}\cN_{-1}'J_{\cN}$ such that $\alpha(\cN)$ becomes the standard representation of $\cN$ on $L^2({\cN})$. 
We denote the preimage of $e_{\cN_{-1}}$ as $e_{-1}\in\cM$ and call it Jones projection for $\cN_{-1}\subset\cN$. 
When $\cN\subset\cM$ admits a downward Jones basic construction, we will simply suppress the isomorphism $\alpha$ and identify $\cM$ with $J_{\cN}\cN_{-1}'J_{\cN}$. 
We use the symbol $\cN_{-1}\subset\cN\subset^{e_{-1}}\cM$ to indicate a downward Jones basic construction with Jones projection $e_{-1}$. 

Let $\cN_{-1}\subset\cN\subset^{e_{-1}}\cM$ be a downward Jones basic construction. Then it is known that all canonical traces (including the trace on $\cM$ as $\End(L^2(\cN)_{\cN_{-1}})$) induced by basic constructions are compatible. 
As a consequence, we restore the Temperley-Lieb relation between Jones projections:
\begin{align*}
    e_{\cN}e_{\cM}e_{\cN} = \delta^{-2}e_{\cN},\quad e_{-1}e_{\cN}e_{-1} = \delta^{-2}e_{-1}.
\end{align*}

Before going into the details, let us outline our proof strategy. 
We will leverage the Jones projection $e_{-1}$ to construct a sequence of partitions of unity in $\cM$ such that the sequence of associated completely positive trace-preserving maps approaches a $*$ homomorphism, which is closely related to the canonical shift introduced by Ocneanu \cite{Ocneanu1989}. 

Let $\cN_{-1}\subset\cN\subset^{e_{-1}}\cM$ be a downward Jones basic construction, then by iterating the Jones basic construction twice we get inclusions
\begin{align*}
    \cN_{-1}\subset\cN\subset^{e_{-1}}\cM\subset^{e_{\cN}}\cM_1\subset^{e_{\cM}}\cM_2,
\end{align*}
with $\cM_2 = \End (L^2(\cM_1)_{\cM})$ acting naturally on $L^2(\cM_1)$. 
Let $e^{\cM}_{\cN_{-1}}$ be the Jones projection on $L^2(\cM)$ with range $L^2(\cN_{-1})$. 
Then according to Theorem 2.6 of \cite{PP1988} (see also Proposition 2.1 in \cite{Bisch1997}), the map $\phi:\delta^2 x e_{\cN}e_{-1}e_{\cM}e_{\cN} y\mapsto xe^{\cM}_{\cN_{-1}}y$ where $x,y\in\cM$ extends to a $*$-isomorphism from $\cM_2$ to $J_{\cM}\cN'_{-1}J_{\cM}$. 
With this identification, the canonical shift $\gamma: \cN'_{-1}\cap\cM\to \cM'\cap\cM_2$ is defined as
\begin{equation}\label{eqn:: canonical shifts}
    \gamma(x) = J_{\cM}J_{\cN}xJ_{\cN}J_{\cM},\quad x\in  \cN'_{-1}\cap\cM.
\end{equation}
It is known that $\gamma$ is a $*$-isomorphism and that $\tau_{\cM_2}\circ\gamma = \tau_{\cM}$. 
For our purpose, we shall consider the inverse of the canonical shift. 
\begin{lemma}\label{lemma:: inverse of canonical shift}
    Let $\cN\subset\cM$ be a finite inclusion of finite von Neumann algebras that admits a  downward Jones basic construction $\cN_{-1}\subset\cN\subset^{e_{-1}}\cM$. 
    Let $\gamma: \cN'_{-1}\cap\cM\to \cM'\cap\cM_2$ be the canonical shift, then the following statements hold: 
    \begin{enumerate}
        \item  For any $x\in \cM'\cap\cM_2$, $\gamma^{-1}(x)$ is the unique element in $\cN_{-1}'\cap \cM$ such that 
        $$\gamma^{-1}(x)e_{\cM}=\delta^4 e_\cM e_\cN e_{-1}xe_{-1}e_\cN e_\cM;$$
        \item For any $\Phi\in \mathbf{CP}_{\cN}(\cM)$ so that $\widehat{\Phi}$ exists, 
        \begin{align*}
            \gamma^{-1}(\widehat{\Phi})=\delta\Phi(e_{-1}).
        \end{align*}
    \end{enumerate}
\end{lemma}
\begin{proof}
    (1): Let $x\in \cM'\cap \cM_2$ be acting on $L^2(\cM_1)$. 
    By viewing $L^2(\cM)$ as the image of $e_{\cM}$, we see $\delta^4 e_\mathcal{M}e_\mathcal{N}e_{-1}xe_{-1}e_\mathcal{N}e_\mathcal{M}$ is an 
    operator on $L^2(\cM)$ commuting with right action of $\cM$. This implies that $\delta^4 e_\mathcal{M}e_\mathcal{N}e_{-1}xe_{-1}e_\mathcal{N}e_\mathcal{M}=ye_{\cM}$ for some $y\in \cM$. Since all operators involved in the expression commute with $\cN_{-1}$, $y\in \cN_{-1}'\cap\cM$. 

    By Theorem 2.11 of \cite{Bisch1997} we know that the canonical shift has the following expression:
    \begin{align*}
        \gamma(y) = \delta^4\sum^n_i \xi_ie_{-1}e_{\cN}e_{\cM} y e_{\cN}e_{-1}\xi^*_i, \quad y\in \cN'_{-1}\cap\cM
    \end{align*}
    where $\{\xi_i\}^n_{i=1}$ is a basis of $\cN$ over $\cN_{-1}$. 
    So pick $x\in \cM'\cap\cM_2$ and let $y\in \cN'_{-1}\cap\cM$ be such that $ye_{\cM} = \delta^4 e_\cM e_\cN e_{-1}xe_{-1}e_\cN e_\cM$, then 
    \begin{align*}
        \gamma(y) &= \delta^4\sum^n_i \xi_ie_{-1}e_{\cN}\left(ye_{\cM}\right)e_{\cN}e_{-1}\xi^*_i\\
        &= \delta^8 \sum^n_i \xi_ie_{-1}e_{\cN}e_{\cM}e_{\cN}e_{-1} x e_{\cN}e_{\cM}e_{\cN}e_{-1}\xi^*_i\\
        &= \sum_{i}\xi_ie_{-1}\xi^*_i x =x.
    \end{align*}
    This proves that $y = \gamma^{-1}(x)$.\\
    % (2): For any $x_1, x_2\in \mathcal{M}'\cap \mathcal{M}_2$, we have
    % \begin{align*}
    %     \theta(x_1)\theta(x_2)e_{\cM} &=\theta(x_1)e_{\cM }\theta(x_2)e_{\cM} \\
    %     &=\delta^8  e_{\cM}e_{\cN} e_{-1} x_1  (e_{-1}e_{\cN} e_{\cM} e_{\cM}e_{\cN} e_{-1}) x_2  e_{-1}e_{\cN} e_{\cM} \\
    %     &=\delta^4 e_{\cM}e_{\cN} e_{-1} x_1e_{-1}x_2  e_{-1}e_{\cN} e_{\cM}\\
    %     &=\delta^4 e_{\cM}e_{\cN} e_{-1} x_1x_2  e_{-1}e_{\cN} e_{\cM}=\theta(x_1x_2)e_{\cM}.
    % \end{align*}
    % For any $x\in \mathcal{M}'\cap \mathcal{M}_2$, we have that $\theta(x)^*=\theta(x^*)$ so $\theta$ is a $*$-homomorphism.\\
    % (3): Since $\cM$ is a Jones basic construction algebra, the trace $\tau_{\cM_2}$ restricts to the trace on $\cM$. 
    % Thus for any $x\in \cM'\cap\cM_2$ we have 
    % \begin{align*}
    %     \tau_\mathcal{M}(\theta(x)) = \tau_{\mathcal{M}_2}(\theta(x))&=\delta^2\tau_{\mathcal{M}_2}(\theta(x)e_{\mathcal{M}})\\
    %     &=\delta^6\tau_{\mathcal{M}_2}(e_\mathcal{M}e_\mathcal{N}e_{-1}xe_{-1}e_\mathcal{N}e_\mathcal{M})\quad \text{(by (ii))}\\
    %     &=\delta^6\tau_{\mathcal{M}_2}(xe_{-1}e_\mathcal{N}e_\mathcal{M}e_\mathcal{N}e_{-1})\\
    %     &=\delta^2\tau_{\mathcal{M}_2}(xe_{-1})\quad \text{by Temperley-Lieb relation}\\
    %     &=\tau_{\mathcal{M}_2}(x).
    % \end{align*}
    (2): By definition $\widehat{\Phi}(e_{\cN}\Omega_{\cM_1})=\displaystyle \frac{1}{\delta}\sum_i\eta_ie_{\cN}\Phi(\eta^*_i)\Omega_{\cM_1}$ for any Pimsner-Popa basis $\{\eta_i\}^m_{i=1}$ of $\cM$ over $\cN$. 
    Thus by (1):
    \begin{align*}
        \gamma^{-1}(\widehat{\Phi}_0)e_{\cM} (m\Omega_{\cM_1}) &= \delta^4 e_{\cM} e_{\cN} e_{-1}\widehat{\Phi} e_{\cN}m\Omega_{\cM_1}\\
        &= \delta^3\sum_{i} e_{\cM} \left(e_{\cN} e_{-1}\eta_ie_{\cN}\Phi(\eta^*_i)\right) m\Omega_{\cM_1}\\
        &= \delta^3\sum_{i} e_{\cM} e_{\cN} \Phi(E_{\cN}(e_{-1}\eta_i)\eta^*_i) m\Omega_{\cM_1}\\
        &= \delta^3 e_{\cM}e_{\cN}\Phi(e_{-1})m\Omega_{\cM_1} = \delta \Phi(e_{-1})m\Omega_{\cM_1}, \quad \forall m\in\cM.
    \end{align*}
    Therefore $\gamma^{-1}(\widehat{\Phi}_0) = \delta \Phi(e_{-1})$.
\end{proof}
% \begin{proposition}\label{prop:: theta is the inverse of the canonical shift}
%     Let $\gamma: \cN'_{-1}\cap\cM\to \cM'\cap\cM_2$ be the canonical shift, then $\theta = \gamma^{-1}$. 
% \end{proposition}
% \begin{proof}
%     From (reference) we know that the canonical shift has the following expression:
%     \begin{align*}
%         \gamma(x) = \delta^4\sum^n_i \xi_ie_{-1}e_{\cN}e_{\cM}xe_{\cN}e_{-1}\xi^*_i,\quad x\in \cN'_{-1}\cap\cM
%     \end{align*}
%     where $\{\xi_i\}^n_{i=1}$ is a basis of $\cN$ over $\cN_{-1}$. 
%     It then follows from the Temperley-Lieb relation and (1) of Lemma \ref{lemma:: inverse of canonical shift} that for $x\in \cN'_{-1}\cap \cM$
%     \begin{align*}
%         \theta(\gamma(x))e_{\cM} &= \delta^8 e_{\cM}e_{\cN}e_{-1}\left(\sum^n_{i=1}\xi_ie_{-1}e_{\cN}e_{\cM}xe_{\cN}e_{-1}\xi^*_i\right)e_{\cN}e_{\cM}\\
%         &= \delta^2 \sum^n_{i=1} e_{\cM}xe_{\cN}E^{\cN}_{\cN_{-1}}(\xi_i)\xi^*_ie_{\cM}\\
%         &= \delta^2 e_{\cM}xe_{\cN}e_{\cM} = xe_{\cM}.
%     \end{align*}
%     Similarly for $y\in \cM'\cap\cM_2$
%     \begin{align*}
%         \gamma(\theta(y)) &= \delta^4\sum^n_i \xi_ie_{-1}e_{\cN}\left(\theta(y)e_{\cM}\right)e_{\cN}e_{-1}\xi^*_i\\
%         &= \delta^8 \sum^n_i \xi_ie_{-1}e_{\cN}e_{\cM}e_{\cN}e_{-1}ye_{\cN}e_{\cM}e_{\cN}e_{-1}\xi^*_i\\
%         &= \sum_{i}\xi_ie_{-1}\xi^*_iy =y.
%     \end{align*}
%     Therefore $\theta$ and $\gamma$ are inverses of each other.
% \end{proof}
\begin{corollary}\label{corollary:: non-spherical}
    Let $\cN\subset \cM$ be a finite inclusion of finite von Neumann algebras, $\cN_{-1}\subset \cN\subset^{e_{-1}} \cM$ be a downward Jones basic construction. 
   Then the followings hold:
    \begin{enumerate}
        \item For any $x,y\in \cN_{-1}'\cap \cN$, $\delta^2 \tau_{\cM}(J_{\cN}x^*J_{\cN}e_{-1}y)=\tau_{\cN}(xy)$;
        \item $\delta^2 E^{\cM}_{\cN'\cap \cM}( e_{-1}) = \Delta_0$ (c.f. Definition \ref{def::Delta_0});
        \item $\delta^2 E^{\cM}_{\cN}(\Delta^{-1/2}_0e_{-1}\Delta^{-1/2}_0)=J_{\cN}\Delta^{-1}_0J_{\cN}$;
        \item $J_{\cN}\Delta^{-1}_0J_{\cN}e_{-1}=\Delta^{-1}_0e_{-1}$;
        \item $\gamma^{-1}(\Delta)=J_{\cN}\Delta^{-1}_0J_{\cN}$.
    \end{enumerate}
\end{corollary}
\begin{proof}
    (1): Choose a Pimsner-Popa basis $\{\xi_j\}_j$ of $\cN$ over $\cN_{-1}$ and identify $\cM$ as $J_{\cN}\cN'_{-1}J_{\cN}$. We have
    \begin{align*}
        \delta^2 \tau_{\cM}(J_{\cN}x^*J_{\cN}e_{-1}y)&=\sum_j\left\langle J_{\cN}x^*J_{\cN}E_{\cN_{-1}}(y\xi_j)\Omega,\xi_j\Omega\right\rangle\\
        &=\sum_j\left\langle E_{\cN_{-1}}(y\xi_j)x\Omega,\xi_j\Omega\right\rangle\\
        &=\sum_j\left\langle xE_{\cN_{-1}}(y\xi_j)\Omega,\xi_j\Omega\right\rangle\quad(\text{since }x\in \cN'_{-1})\\
        &=\sum_j\left\langle xE_{\cN_{-1}}(y\xi_j)\xi^*_j\Omega,\Omega\right\rangle\\
        &=\tau_{\cN}(xy).
    \end{align*}
    (2): By (1), we have that for any $x\in J_{\cN}\cN J_{\cN}\cap \cM$, 
    \begin{align*}
        \delta^2 \tau_{\cM}\left(x E^{\cM}_{\cN' \cap \cM}(e_{-1})\right) &= \delta^2 \tau_{\cM}(xe_{-1})\\
        &=\tau_{\cN}(J_{\cN} x^*J_{\cN})\\
        &=\tau_{J_{\cN} \cN J_{\cN}}(x)
    \end{align*}
    This implies that $\delta^2 E^{\cM}_{\cN' \cap \cM}(e_{-1}) \in\cN'\cap\cM$ is the Radon-Nikodym derivative between $\tau_{\cM}$ and $\tau_{J_{\cN}\cN J_{\cN}}$.\\
    (3): We see that $\delta^2 E_{\cN}(\Delta^{-1/2}_0e_{-1}\Delta^{-1/2}_0)\in \cN_{-1}'\cap \cN$, and for all $y\in \cN_{-1}'\cap \cN$, 
    \begin{align*}
        &\delta^2 \tau_{\cM}(\Delta_0^{-1/2}e_{-1}\Delta_0^{-1/2}y)\\
        &=\delta^2 \tau_{\cM}(\Delta_0^{-1}e_{-1}y)\\
        &=\tau_N(J_{\cN} \Delta^{-1}_0 J_{\cN} y) \quad \text{followed from (1)}.
    \end{align*}
    Then $\delta^2 E^{\cM}_{\cN}(\Delta^{-1/2}_0e_{-1}\Delta^{-1/2}_0)=J_{\cN}\Delta^{-1}_0J_{\cN}$.\\
    (4): This follows from the fact that $J_{\cN}\Delta^{-1}_0J_{\cN}$ and $\Delta^{-1}_0$ commute with $\cN_{-1}$.\\
    (5): By Definition \ref{def::Delta_0}, we see $\Delta^{-1}_0 = J_{\cM}\Delta J_{\cM}$.
    Therefore
    \begin{align*}
        J_{\cN}\Delta^{-1}_0 J_{\cN} = J_{\cN}J_{\cM}\Delta J_{\cM}J_{\cN} = \gamma^{-1}(\Delta),
    \end{align*}
    completing the proof. 
\end{proof}

\begin{theorem}\label{thm::Formula for entropy between bimodule CP maps}
    Let $\mathcal{N}\subset\mathcal{M}$ be a finite inclusion of finite von Neumann algebras and
    $\Phi\preccurlyeq \Psi:\cM\to\cM$ are completely positive $\cN$-bimodule maps.
    If the inclusion admits a downward Jones basic construction $\cN_{-1}\subset \cN\subset^{e_{-1}} \cM$, then
    \begin{align}\label{eqn::entropy in terms of multiplier}
        H(\Phi|\Psi) = \delta D_{\tau_{\cM_2}}(\Delta^{1/2}\widehat{\Phi}\Delta^{1/2}| \Delta^{1/2}\widehat{\Psi}\Delta^{1/2}).
    \end{align}
\end{theorem}
\begin{proof}
    We begin by constructing the partition of unity. 
    Let $e_{-1}\in\cM$ be a Jones projection. 
    By (2) of Corollary \ref{corollary:: non-spherical}, we have $\delta^2E^{\cM}_{\cN'\cap \cM}(\Delta^{-1/2}_0 e_{-1}\Delta^{-1/2}_0)=1$. 
    Then by the relative Dixmier property for finite inclusions \cite{Popa1999}, for each $\epsilon>0$ we can take a set of $n$ unitaries $\{u_k\}^n_{k=1 }$ in $\cN$ such that
    \begin{align*}
        \frac{1-\epsilon}{1+\epsilon}\leq\frac{\delta^2 }{n(1+\epsilon)}\sum^{n}_{k=1}u_k\Delta^{-1/2}_0e_{-1}\Delta^{-1/2}_0u^*_k\leq 1.
    \end{align*}
    Put $\displaystyle x_k=\frac{\delta^2 }{n(1+\epsilon)}u_k\Delta^{-1/2}_0e_{-1}\Delta^{-1/2}_0u^*_k$ for $1\leq k\leq n$ and 
    $x_{n+1}=1-\sum^n_{k=1}x_k$. 
    Let $T:\cM'\cap\cM_2\to\cM$ be the completely positive trace-preserving map constructed in the proof of Theorem \ref{thm::popaentropycompletepositive} associated to the partition $\{x_{k}\}^{n+1}_{k=1}$, then $ \delta T(\Delta^{1/2}\widehat{\Phi}\Delta^{1/2}) = \big(y_k \big)^{n+1}_{k=1}$ where 
    \begin{align*}
        y_k = \frac{\delta^2}{n(1+\epsilon)}\Phi\left(u_k\Delta^{-1/2}_0e_{-1}\Delta^{-1/2}_0u^*_k\right),y_{n+1} = \Phi(x_{n+1}),\quad 1\leq  k\leq n.
    \end{align*}
    For each $1\leq k\leq n$, we have
    \begin{align*}
        y_k &= \frac{\delta^2}{n(1+\epsilon)}\Phi\left(u_k\Delta^{-1/2}_0e_{-1}\Delta^{-1/2}_0u^*_k\right)\\
        &= \frac{\delta^2}{n(1+\epsilon)}u_k\Phi\left(J_{\cN}\Delta^{-1/2}_0J_{\cN}e_{-1}J_{\cN}\Delta^{-1/2}_0J_{\cN}\right)u^*_k \\
        &= \frac{\delta^2}{n(1+\epsilon)}u_k\gamma^{-1}(\Delta^{1/2})\Phi(e_{-1})\gamma^{-1}(\Delta^{1/2})u^*_k\\
        &= \frac{\delta}{n(1+\epsilon)}u_k\gamma^{-1}(\Delta^{1/2}\widehat{\Phi}\Delta^{1/2})u^*_k ,
    \end{align*}
    where in the third and the last equality we used Corollary \ref{corollary:: non-spherical} and Lemma \ref{lemma:: inverse of canonical shift}. 
    Therefore
    \begin{align*}
	    H(\Phi|\Psi)&\geq \delta D(T(\Delta^{1/2}\widehat{\Phi}\Delta^{1/2})\|T(\Delta^{1/2}\widehat{\Psi}\Delta^{1/2}))\\
	    &\geq \frac{\delta}{n(1+\epsilon)}\sum^n_{k=1}D_{\tau_{\cM}}\left(u_k\gamma^{-1}(\Delta^{1/2}\widehat{\Phi}\Delta^{1/2})u^*_k\|u_k\gamma^{-1}(\Delta^{1/2}\widehat{\Psi}\Delta^{1/2})u^*_k\right)\\
	    &=\frac{\delta}{1+\epsilon}D_{\tau_{\cM_2}}(\Delta^{1/2}\widehat{\Phi}\Delta^{1/2}\| \Delta^{1/2}\widehat{\Psi}\Delta^{1/2}).
    \end{align*}
    Note that in the las equality we used $\tau_{\cM}\circ\gamma^{-1}=\tau_{\cM_2}$. 
    Now taking $\epsilon\rightarrow 0$, we see that the theorem is true. 
\end{proof}

\begin{theorem}\label{thm:: formula for H(M|N) in terms of Jones projection}
    Let $\mathcal{N}\subset\mathcal{M}$ be a finite inclusion of finite von Neumann algebras with a downward Jones basic construction $\cN_{-1}\subset\cN\subset^{e_{-1}}\cM$. 
    Then for any $\Phi,\Psi\in \mathbf{CP}_{\cN}(\cM)$ with $\Phi\preccurlyeq \Psi$,
\begin{align*}
H(\Phi|\Psi)=\delta^2 D_{\tau_{\cM}}(\Phi(\Delta_0^{-1/2}e_{-1}\Delta_0^{-1/2})\|\Phi(\Delta_0^{-1/2}e_{-1}\Delta_0^{-1/2})).
\end{align*}
\end{theorem}
\begin{proof} 
    Applying $\gamma^{-1}$ to the right hand side of Equation \eqref{eqn::entropy in terms of multiplier} and by $\tau_{\cM}\circ\gamma^{-1}=\tau_{\cM_2}$ we obtain:
    \begin{align*}
        &H(\Phi|\Psi)\\
        &=\delta D_{\tau_{\cM_2}}(\Delta^{1/2}\widehat{\Phi}\Delta^{1/2}\| \Delta^{1/2}\widehat{\Psi}\Delta^{1/2})\\
        &= 
        \delta D_{\tau_{\cM}} \left(\gamma^{-1}(\Delta^{1/2}\widehat{\Phi}\Delta^{1/2})\|\gamma^{-1}(\Delta^{1/2}\widehat{\Phi}\Delta^{1/2})\right)\\
        &= \delta D_{\tau_{\cM}} \left((J_{\cN}\Delta^{-1/2}_0J_{\cN})\gamma^{-1}(\widehat{\Phi})(J_{\cN}\Delta^{-1/2}_0J_{\cN})\|(J_{\cN}\Delta^{-1/2}_0J_{\cN})\gamma^{-1}(\widehat{\Psi})(J_{\cN}\Delta^{-1/2}_0J_{\cN})\right)\\
        &= \delta^2 D_{\tau_{\cM}} \left( \Phi(\Delta^{1/2}_0e_{-1}\Delta^{1/2}_0)\| \Psi(\Delta^{1/2}_0e_{-1}\Delta^{1/2}_0)\right),
    \end{align*}
    Notice that the last equality is implied by Lemma \ref{lemma:: inverse of canonical shift} and Corollary \ref{corollary:: non-spherical}. 
\end{proof}
We now consider an inclusion $\cN\subset\cM$ of finite dimensional $C^*$ algebras and try to determine the necessary and sufficient condition for equality
\begin{align*}
    H(\cM\vert\cN) = \delta D_{\tau_{\cM_2}}(\Delta^{1/2}\widehat{id_{\cM}}\Delta^{1/2}| \Delta^{1/2}\widehat{E_{\cN}}\Delta^{1/2})
\end{align*} 
We start with the computation of the second term, using the formula obtained in Section 3.2.

We adopt the notations from Section 6 of \cite{PP1986}. 
Let $K$ and $L$ be two index sets of finite cardinals. 
$\cN$ and $\cM$ will be described as 
\begin{align*}
    \cN = \bigoplus_{k\in K}M_{n_k}(\mathbb{C}),\quad  \cM = \bigoplus_{l\in L}M_{m_l}(\mathbb{C}).
\end{align*} 
The inclusion is described by the adjacent matrix $A = (a_{kl})_{k\in K,l\in L}$. 
We have the dimension (row) vectors $\vec{n} = (n_k)_{k\in K}$, $\vec{m} = (m_l)_{l\in L}$ and trace (column) vectors $\vec{s}=(s_k)_{k\in K}$, $\vec{t}=(t_l)_{l\in L}$ for $\cN$, $\cM$ respectively. 
They satisfy the relation
\begin{align*}
    \vec{n}A = \vec{m},\quad A\vec{t} = \vec{s}.
\end{align*}

For $k\in K$ and $l\in L$, denote by $e_k$ ($f_l$) the minimal central projection of $\cN$ ($\cM$), then $e_kf_l$ are the minimal central projections of $\cN'\cap\cM$. 
Note that $e_kf_l$ is a rank $n_ka_{kl}$ subprojection of $f_l$, so we have
\begin{align*}
    \tau_{\cM}(e_kf_l) = n_ka_{kl}t_l,\quad k\in K,l\in L.
\end{align*}

To compute $\tau_{\cN'}$, we construct a set of Pimsner-Popa basis of $\cM$ as follows. 
For a fixed $l$, we decompose $\cM f_l$ as
\begin{align*}
    \cM f_l = \bigoplus_{(k_1,k_2)\in K\times K} e_{k_1}\cM f_le_{k_2}.
\end{align*}
Denote $e_{k_1}\cM f_le_{k_2}$ as $B^l_{k_1,k_2}$. 
Each $B^l_{k_1,k_2}$ will be identified with  $M_{n_{k_1}\times n_{k_2}}(\mathbb{C})\otimes M_{a_{k_1 l}\times a_{k_2 l}}(\mathbb{C})$ in the way such that for any $b\otimes c\in B^l_{k_1,k_2}$ we have
\begin{align*}
    x\left( b\otimes c\right)y = x_{k_1}by_{k_2}\otimes c,
\end{align*}
where $x = \bigoplus_{k\in K}x_k$ and $y = \bigoplus_{k\in K}y_k$ are in $\cN$. 
Therefore $B^l_{k_1,k_2}$, as an $\cN e_{k_1}$-$\cN e_{k_2}$-bimodule, decomposes into the direct sum of $a_{k_1 l}a_{k_2 l}$ irreducibles. 
Identity $\cN e_{k}$ as $M_{n_k}(\mathbb{C})$, we see that each irreducible submodule of $B_{k_1,k_2}$ is isomorphic to
\begin{align*}
    _{M_{n_{k_1}}(\mathbb{C})}M_{n_{k_1}\times n_{k_2}}(\mathbb{C})_{M_{n_{k_2}}(\mathbb{C})},
\end{align*}
with the bimodule structure given by matrix multiplications. 
By considering $M_{n_{k_1}\times n_{k_2}}(\mathbb{C})$ as right $\cN e_{k_2}$-module, it further decomposes into $n_{k_1}$ copies of irreducibles of $\cN e_{k_2}$. 

Now it is easy to produce a set of Pimsner-Popa basis for $L^2(\cM)_{\cN}$. 
For $(k,l)\in K\times L$, choose a set of basis of $\bigoplus_{k'\in K} B^l_{k',k}$ as a right $\cN$-module as $\{\xi_{i,k',k,l}\}^{n_{k'}a_{k' l}a_{k l}}_{i=1}$ so that each $\xi_{i,k',k,l}$ generates an irreducible right module of $\cN e_{k}$ (thus of $\cN$). 
By properly scaling each $\xi_{i,k',k,l}$, we can assume that $E^{\cM}_{\cN}(\xi^*_{i,k',k,l}\xi_{i,k',k,l})$ is a minimal projection under $e_k$. 
Then we can compute $\tau_{\cN'}(e_kf_l)$ as
\begin{align*}
    \tau_{\cN'}(e_kf_l) &= \delta^{-2}\sum_{k'\in K}\sum_{1\leq i\leq n_{k'}a_{k' l}a_{k l}}\tau_{\cM}(\xi_{i,k',k,l}e_kf_l\xi^*_{i,k',k,l})\\
    &= \delta^{-2}\sum_{k'\in K} n_{k'}a_{k' l}a_{k l}s_k = \delta^{-2}s_ka_{kl}m_l.
\end{align*}
Therefore we obtain:
\begin{align*}
    \Delta_0 = \sum_{k\in K,l\in L}\frac{s_ka_{kl}m_l}{\delta^2 n_ka_{kl}t_l}e_kf_l = \sum_{k\in K,l\in L}\frac{s_km_l}{\delta^2 n_kt_l}e_kf_l.
\end{align*}
Inserting it into the formula obtained in Proposition \ref{prop:: explicit upperbound}:
\begin{align*}
    \delta D_{\tau_{\cM_2}}(\Delta^{1/2}\widehat{id_{\cM}}\Delta^{1/2}\| \Delta^{1/2}\widehat{E_{\cN}}\Delta^{1/2}) &= 2\log\delta -(-\tau_{\cM}(\log\Delta_0))\\
    &= 2\log\delta -\sum_{k\in K,l\in L}n_ka_{kl}t_l \log\frac{\delta^2 n_kt_l}{s_km_l}\\
    &= \sum_{k\in K,l\in L}n_ka_{kl}t_l \log\frac{s_km_l}{n_kt_l}.
\end{align*}
By Theorem 6.2 of \cite{PP1986}, we have
\begin{align*}
    H(\cM\vert\cN) &= \sum_{l\in L}m_lt_l \log\frac{m_l}{t_l} + \sum_{k\in K}n_ks_k \log \frac{s_k}{n_k} + \sum_{k\in K,l\in L}n_ka_{kl}t_l\log\min\left\{\frac{n_k}{a_{kl}},1\right\}\\
    &= \sum_{k\in K,l\in L}n_ka_{kl}t_l \log\frac{s_km_l}{n_k t_l} + \sum_{k\in K,l\in L}n_ka_{kl}t_l\log\min \left\{\frac{n_k}{a_{kl}},1\right\},
\end{align*}
where in the second equality we use that $m_l = \sum_{k\in K}n_ka_{kl}$ and $s_k = \sum_{l\in L} a_{kl}t_l$, $l\in L$ and $k\in K$. 
Thus we obtain
\begin{align*}
    &\delta D_{\tau_{\cM_2}}(\Delta^{1/2}\widehat{id_{\cM}}\Delta^{1/2}\| \Delta^{1/2}\widehat{E_{\cN}}\Delta^{1/2})-H(\cM\vert\cN)\\
    &= -\sum_{k\in K,l\in L}n_ka_{kl}t_l\log\min\left\{\frac{n_k}{a_{kl}},1\right\}\\
    &= \sum_{k\in K,l\in L}n_ka_{kl}t_l\log\max\left\{\frac{a_{kl}}{n_k},1\right\}.
\end{align*}
In all, we have proved the following proposition.
\begin{proposition}\label{prop:: case of equality in finite dimensions}
    Let $\cN\subset\cM$ be an inclusion of finite dimensional $C^*$ algebras with inclusion matrix $[a_{kl}]$ and dimension vector $\Vec{n}= {n_k}_{k\in K}$ for $\cN$. 
    Then the necessary and sufficient condition for the equality 
\begin{align*}
    H(\cM\vert\cN) = \delta D_{\tau_{\cM_2}}(\Delta^{1/2}\widehat{id_{\cM}}\Delta^{1/2}\| \Delta^{1/2}\widehat{E_{\cN}}\Delta^{1/2})
\end{align*}
to hold is that 
\begin{align*}
    a_{kl}\leq n_k,\quad k\in K,l\in L.
\end{align*}
\end{proposition}
\begin{remark}
    This condition can be related to downward Jones basic construction as follows. 
    Consider the inclusion $\mathbb{C}\subset \cN \subset \mathcal{B}(L^2(\cN))$ which is a basic construction. 
    By characterization of finite dimensional basic construction as in \cite{Jones1983}, the adjacent matrix for the inclusion $\cN \otimes 1\subset \mathcal{B}(L^2(\cN))\otimes \mathbb{C}^{|L|}$ is $\tilde{A} = (\tilde{a}_{kl})_{k\in K,l\in L}$, $\tilde{a}_{kl} = n_k$. 
    Therefore the condition $a_{kl}\leq n_k$ for all $k$ and $l$ is equivalent to the existence of a projection $p'\in \cN'\otimes \mathbb{C}^{|L|}$ with central support $1$ such that the inclusion $\cN\subset\cM$ is isomorphic to 
    \begin{align*}
        (\cN \otimes 1) p'\subset p'(\mathcal{B}(L^2(\cN))\otimes \mathbb{C}^{|L|})p'.
    \end{align*}
\end{remark} 
\section{Derivatives for Completely Positive Maps}
In this section, we study the derivative between comparable completely positive maps as a generalization of Fourier multiplier. 
Our aim is to establish that Fourier multipliers can be regarded as derivatives with respect to a conditional expectation. 

With the notion of relative tensor product of bimodules, we derive a formula expressing the derivative of the composition of completely positive maps in terms of their derivatives. 
This formula will be instrumental in proving the monotonicity of relative entropy in the subsequent section. 

The following lemma, though simple, is essential for us. 
It has been observed in many cases, see for instance \cite{Connes1994}, \cite{Popa2010}.
\begin{lemma}\label{lemma:RN}
    Suppose $\Psi,\Phi:\mathcal{A}\rightarrow\mathcal{B}$ are normal completely positive maps.
    Then the following are equivalent:
    \begin{itemize}
    \item[(1)] $\Phi\preccurlyeq \Psi$;
    \item[(2)] there is a unique positive element $h\in \End(_\mathcal{A}\mathcal{H}^\Psi_\mathcal{B})$ such that for all $a\in\mathcal{A}$
    $$\Phi(a)=v_\Psi^*\pi_\Psi(a)hv_\Psi.$$
    \end{itemize}
\end{lemma}  
\begin{proof}
    (1)$\implies$(2): Assume $c\Psi-\Phi$ is completely positive for some $c>0$.
    Suppose $\varphi$ is a normal faithful state on $\mathcal{B}$.
    For any $\displaystyle \zeta=\sum^n_{i=1} a_i\otimes \xi_i\in \mathcal{A}\otimes L^2(\mathcal{B}, \varphi)$, we denote by $[\zeta]_{\Phi}, [\zeta]_\Psi$ the image of $\zeta$ in $\mathcal{H}^\Phi$, $\mathcal{H}^\Psi$.
    Note that
    \begin{align*}
        \left\langle [\zeta]_\Phi,[\zeta]_{\Phi}\right\rangle_{\Phi}=\sum_{i,j=1}^n \left\langle\Phi(a_i^*a_j)\xi_j,\xi_i\right\rangle.
    \end{align*}
    With $A=(a^*_ia_j)_{i,j=1}^n\in \big(\mathcal{A}\otimes M_n(\mathbb{C})\big)_+$ and $\widetilde{\xi}=\begin{pmatrix}
        \xi_1 \\
        \vdots \\
        \xi_n
    \end{pmatrix}\in L^2(B, \varphi)\otimes \mathbb{C}^n$, we have
    \begin{align*} 
    \left\langle[\zeta]_\Phi,[\zeta]_\Phi\right\rangle_{\Phi}
    =\left\langle\Phi(A)\widetilde{\xi},\widetilde{\xi}\right\rangle
    \leq c \left\langle\Psi(A)\widetilde{\xi},\widetilde{\xi}\right\rangle
    =c\left\langle [\zeta]_\Psi,[\zeta]_\Psi\right\rangle_{\Psi}.
    \end{align*}
By the fact that $\mathcal{A}\otimes L^2(\mathcal{B}, \varphi)$ is dense in $\mathcal{H}^\Psi$, $\mathcal{H}^\Phi$ respectively, the linear map $u:\mathcal{H}^\Psi\to \mathcal{H}^\Phi$ defined by $u[\zeta]_\Psi=[\zeta]_\Phi$ for any $\zeta\in \mathcal{A}\otimes L^2(\mathcal{B}, \varphi)$ is a bounded bimodule map and $\|u\|\leq c^{1/2}$. 
By a direct check, we see that $h=u^*u\in \End(_\mathcal{A}\mathcal{H}^\Psi_\mathcal{B})$.
Note that for any $\xi\in L^2(\mathcal{B}, \varphi)$,
\begin{align*}
    v_\Phi\xi
    =[1_{\mathcal{A}}\otimes \xi]_{\Phi}
    =u[1_{\mathcal{A}}\otimes \xi]_{\Psi} 
    =u v_\Psi \xi.
\end{align*}
We see that $uv_\Psi=v_\Phi$.
Now, we obtain that for any $a\in\mathcal{A}$,
\begin{align*}
    \Phi(a)=&v_{\Phi}^*\pi_{\Phi}(a)v_{\Phi}
    =v_\Psi^*u^*\pi_{\Phi}(a)uv_\Psi\\
    =&v_\Psi^*u^*u\pi_{\Psi}(a)v_\Psi
    =v_\Psi^*h\pi_{\Psi}(a)v_\Psi
\end{align*}

We shall prove the uniqueness of $h$.
Suppose $k$ is another positive operator in $\End(_\mathcal{A}\mathcal{H}(\Psi)_\mathcal{B})$ such that $\Phi(a)=v_\Psi^*k\pi_\Psi(a)v_\Psi$.
Then we have that for any $a_1, a_2\in\mathcal{A}$ and $\xi_1, \xi_2\in L^2(\mathcal{B}, \varphi)$,
\begin{align*}
    \left\langle k^{1/2}[a_1\otimes\xi_1]_\Psi, k^{1/2} [a_2\otimes\xi_2]_\Psi\right\rangle_\Psi
    =&\left\langle \Phi(a_2^*a_1)\xi_1, \xi_2\right\rangle_\varphi \\
    =&\left\langle h^{1/2}[a_1\otimes\xi_1]_\Psi, h^{1/2}[a_2\otimes\xi_2]_\Psi\right\rangle_\Psi.
\end{align*}
Hence $k=h$.

(2)$\implies$ (1): Note that the operator $\|h\|_{\infty}-h$ is positive. We see that
\begin{align*}
    (\|h\|_{\infty}\Psi-\Phi)(\cdot)=v_\Psi^*(\|h\|_{\infty}-h)\pi_\Psi(\cdot)v_\Psi
\end{align*}
is a completely positive.
This completes the proof.
\end{proof}
\begin{definition}\label{def::2}
    Suppose $\Psi,\Phi:\mathcal{A}\to\mathcal{B}$ normal completely positive with $\Phi \preccurlyeq\Psi$, the unique positive element in $\text{End}(_\mathcal{A}\mathcal{H}^{\Psi}_\mathcal{B})$ which satisfies $(2)$ of Lemma \ref{lemma:RN} will be called the derivative of $\Phi$ with respect to $\Psi$, and is denoted as $h_{\Phi,\Psi}$.
\end{definition}
Given a normal completely positive map $\Phi:\mathcal{A}\to\mathcal{B}$ and a faithful normal state $\varphi$ on $\mathcal{B}$, put $\varphi_0=\varphi\circ\Phi$. 
Then by Kadison-Schwarz inequality, there exists $M>0$ such that
\begin{align*}
    \varphi(\Phi(a)^*\Phi(a))\leq M\varphi(\Phi(a^*a)),\quad a\in\mathcal{A}.
\end{align*}
It then follows that
\begin{align*}
    a\Omega_{\varphi_0}\to \Phi(a)\Omega_{\varphi},\quad a\in\mathcal{A}
\end{align*}
extends to a bounded linear map from $L^2(\mathcal{A},\varphi_0)$ to $L^2(\mathcal{B},\varphi)$. 
We shall adopt the notion from \cite{Petz1988} and denote this bounded linear map as $V^{\varphi}_{\Phi}$. 
Here we show that $V^{\varphi}_{\Phi}$ can be recovered from the derivative $h_{\Phi,\Psi}$.
\begin{definition}
    Given normal completely positive map $\Psi:\mathcal{A}\to\mathcal{B}$ and a faithful state $\varphi$ on $\mathcal{B}$, set $\varphi_0=\varphi\circ\Psi$. 
    Define the isometry $u_{\Psi}:L^2(\mathcal{A},\varphi_0)\to \mathcal{H}^{\Psi}$ such that 
    \begin{align*}
        u_{\Psi}:  a\Omega_{\varphi_0}\mapsto a\Omega_{\Psi,\varphi},\quad a\in \mathcal{A}.
    \end{align*}
\end{definition}
\begin{proposition}\label{prop:: The derivative implements Phi}
    For every $\Phi\preccurlyeq \Psi:\mathcal{A}\to\mathcal{B}$ and every faithful state $\varphi$ on $\mathcal{B}$, 
    \begin{align*}
        V^{\varphi}_{\Phi}=v^*_{\Psi,\varphi}h_{\Phi,\Psi}u_{\Psi,\varphi}.
    \end{align*}
    Moreover, $h_{\Phi,\Psi}$ is the unique positive operator in $\End(_{\mathcal{A}}\mathcal{H}^{\Psi}_{\mathcal{B}})$ which satisfies the above equation.
\end{proposition}
\begin{proof}
    For each $a\in\mathcal{A}$, we have
    \begin{align*}
    v^*_{\Psi,\varphi}h_{\Phi,\Psi}\left(u_{\Psi,\varphi}a\Omega_{\varphi_0}\right) &= v^*_{\Psi,\varphi} h_{\Phi,\Psi}a\Omega_{\Psi,\varphi}\\
        &=\left(v^*_{\Psi,\varphi} h_{\Phi,\Psi}\pi_{\Psi}(a)v_{\Psi,\varphi}\right)\Omega_{\varphi}\\
        &= \Phi(a)\Omega_{\varphi}.
    \end{align*}
    Hence $V^{\varphi}_{\Phi}=v^*_{\Psi,\varphi}h_{\Phi,\Psi}u_{\Psi,\varphi}$. 
    Clearly for positive element $k\in \End(_{\mathcal{A}}\mathcal{H}^{\Psi}_{\mathcal{B}})$  satisfying $V^{\varphi}_{\Phi}=v^*_{\Psi,\varphi}k u_{\Psi,\varphi}$ we have 
    $v^*_{\Psi,\varphi}k\pi_{\Psi}(a)v_{\Psi,\varphi}=\Phi(a)$ for every $a\in\mathcal{A}$. 
    Therefore by Lemma \ref{lemma:RN} $k=h_{\Phi,\Psi}$.
\end{proof}

\begin{remark}
    Notice that $v_{\Psi}$ and $u_{\Psi}$ are left and right multiplications by the bounded vector $\Omega_{\Psi}$ in $_{\mathcal{A}}\mathcal{H}^{\Psi}_{\mathcal{B}}$.
\end{remark}

\begin{notation}
    When $\cN\subset\cM$ is a finite inclusion and $\Phi\in \mathbf{CP}_{\cN}(\cM)$ is completely positive we have $\Phi\preccurlyeq E_{\cN}$ by Remark \ref{rmk:: E_N dominates bimodule maps}, and we shall abbreviate $h_{\Phi,E_{\cN}}$ as $h_{\Phi}$. 
\end{notation}
Now we discuss the connection between derivative for completely positive bimodule maps and their Fourier multipliers. 
First let us remark that, using planar algebra, for $\Phi\in\mathbf{CP}_{\cN}(\cM)$ we have
\begin{align*}
    V^{\tau_{\cM}}_{\Phi} = \vcenter{\hbox{
        \begin{tikzpicture}
          \begin{scope}[scale=0.5]% 2-box
          \FillGrey (0.3,-0.5) -- (1,-0.5)
          -- (1,1.5) -- (0.3,1.5);
          \VeryThinLine (0.3,-0.5) -- (0.3,1.5);
          \VeryThinLine (1,-0.5) -- (1,1.5);
          \draw[fill=white] (0,0) rectangle (1.3,1);
          \node at (0.65,0.5) {$\Phi$};
          \end{scope}
        \end{tikzpicture}
      }}.
\end{align*}
Since $\Phi\preccurlyeq E_{\cN}$ by the result of \cite{HJJLW23}, its the derivative $h_{\Phi}$ is well-defined. 
So we are in a position to apply Proposition \ref{prop:: The derivative implements Phi}. 
Along with this, we obtain a generalization of the Pimsner-Popa inequality (c.f. Theorem 2.2 of \cite{PP1986}) for completely positive bimodule maps.
% \begin{lemma}
%     Let $\cN_{-1}\subset\cN\subset\cM$ be a downward Jones basic construction . 
%     Then any positive operator in $\cM$ is of the form $\displaystyle \sum^m_{i=1}n^*_ie_{-1}n_i$ with $n_i\in \cN$.
% \end{lemma}
% \begin{proof}
%     Since $\cM$ is the Jones basic construction of $\cN_{-1}\subset\cN$, any element $x\in\cM$ is of the form $\sum^m_{i=1}a_ie_{-1}b_i$ where $a_i,b_i\in\cN$. Therefore 
%     \begin{align*}
%         x^*x=\sum^m_{i,j=1} b^*_iE_{\cN_{-1}}(a^*_ia_j)e_{-1}b_j.
%     \end{align*}
%     We now define $\cN$-valued matrices $\mathbf{a}=
%     \begin{bmatrix}
%         a_1 & \cdots & a_m
%     \end{bmatrix}$ and $\mathbf{b}=
%     \begin{bmatrix}
%         b_1\\
%         \vdots\\
%         b_m
%     \end{bmatrix}$. 
%     Then 
%     \begin{align*}
%         \sum^m_{i,j=1} b^*_iE_{\cN_{-1}}(a^*_ia_j)e_{-1}b_j &=\mathbf{b}^*E_{\cN}(\mathbf{a}^*\mathbf{a})(e_{-1}\otimes I_m)\mathbf{b}\\
%         &= \mathbf{b}^*E_{\cN}(\mathbf{a}^*\mathbf{a})^{1/2}(e_{-1}\otimes I_m)E_{\cN}(\mathbf{a}^*\mathbf{a})^{1/2}\mathbf{b}.
%     \end{align*}
%     Now set $n_i$ to be the ith entry in $E_{\cN}(\mathbf{a}^*\mathbf{a})^{1/2}\mathbf{b}$, then $n_i\in\cN$ and $x^*x=\sum^m_i n^*_ie_{-1}n_i$ as desired.
% \end{proof}
\begin{proposition}\label{prop:: norm of Fourier multiplier}
    Let $\cN\subset\cM$ be a finite inclusion of finite von Neumann algebras that admits a downward Jones basic construction $\cN_{-1}\subset\cN\subset^{e_{-1}}\cM$. 
    Suppose $\Phi\in \mathbf{CP}_{\cN}(\cM)$.
    Then 
    \begin{align*}
        \delta^2\|\Phi(e_{-1})\| _{\infty}=\inf\{c>0|cE_{\cN}-\Phi\text{ is completely positive}\}.
    \end{align*}
\end{proposition}
\begin{proof}
    We first prove that the map $\delta^2\|\Phi(e_{-1})\|_{\infty}E_{\cN}-\Phi$ is positive. 
    Recall that by the matrix trick as in Proposition 2.1 of \cite{PP1986}, any positive operator in $\cM$ is of the form $a=\sum^m_{i=1}n^*_ie_{-1}n_i$ with $n_i\in \cN$. Therefore
    \begin{align*}
        \Phi(a)&=\sum^m_{i=1}\Phi(n^*_ie_{-1}n_i)=\sum^m_{i=1}n^*_i\Phi(e_{-1})n_i\\
        &\leq\|\Phi(e_{-1})\|_{\infty}\sum^m_{i=1}n^*_in_i\\
        &=\delta^2\|\Phi(e_{-1})\|_{\infty}\sum^m_{i=1}n^*_iE_{\cN}(e_{-1})n_i\\
        &=\delta^2\|\Phi(e_{-1})\|_{\infty}E_{\cN}(a),
    \end{align*}
    so $\delta^2\|\Phi(e_{-1})\|_{\infty}E_{\cN}-\Phi$ is positive. 

    Now consider the case $\Phi\otimes id_{n}$, which is a completely positive $\cN\otimes M_{n}(\mathbb{C})$-bimodule map on $\cM\otimes M_{n}(\mathbb{C})$. 
    We notice that $e_{-1}\otimes I_{n}$ is a Jones projection in $\cM\otimes M_{n}(\mathbb{C})$. 
    Thus by replacing $e_{-1}$ with $e_{-1}\otimes I_{n}$ and choosing $n_i$ to be in $\cN\otimes M_{n}(\mathbb{C})$, the above argument applies also to $\left( \delta^2\|\Phi(e_{-1})\|_{\infty}E_{\cN}-\Phi\right)\otimes id_{n}$, showing that it is a positive map for each $n$.

    Since the tensor product preserves downward Jones basic construction, a $\cN$-$\cN$-bimodule map on $\cM$ is positive if and only if it is completely positive. 
 Moreover, we have  
    \begin{align*}
        &\inf\{c>0|cE_{\cN}-\Phi\text{ is completely positive}\}\\
        &= \inf\{c>0|cE_{\cN}-\Phi\text{ is positive}\}\\
        &= \inf\{c>0|cE_{\cN}(e_{-1})-\Phi(e_{-1})\geq 0\}\\
        &= \delta^2\|\Phi(e_{-1})\|_{\infty}.
    \end{align*}
    This completes the proof.
\end{proof}
\begin{remark}
    Notice that the Pimsner-Popa inequality is a special case of this proposition with $\Phi=id_{\cM}$. 
    That is, the constant
    \begin{align*}
        \delta^{-2}\|id_{\cM}(e_{-1})\|_{\infty}=\delta^{-2}=\lambda(\cM,\cN)
    \end{align*}
    is the largest among all $\lambda>0$ such that $\lambda\cdot id_{\cM}\preccurlyeq E_{\cN}$. 
    Here $\lambda(\cM,\cN)$ is called the Pimsner-Popa index.
\end{remark}
\begin{theorem}\label{thm:: Fourier multiplier is derivative}
    Let $\cN\subset\cM$ be a finite inclusion of finite von Neumann algebras. 
    Suppose $\Phi\in \mathbf{CP}_{\cN}(\cM)$.
    Then under the natural isomorphism between 
    $\End(_{\cM}\mathcal{H}^{E_{\cN}}_{\cM})$ and $\cM'\cap\cM_2$ from Proposition \ref{prop::bimodule of a condiitonal expectation}, we have
\begin{align*}
    h_{\Phi}=\delta\widehat{\Phi}.
\end{align*}
\end{theorem}
\begin{proof}
    Since $\Phi\preccurlyeq E_{\cN}$, the derivative $h_{\Phi}$ exists. 
    Recall that in $\mathscr{P}^{\cN\subset\cM}$ the tangle 
    $\vcenter{\hbox{\begin{tikzpicture}[scale=0.6]
        \begin{scope}% Fourier transform of 2-box
        \ThinLine (1.6,0)--(1.6, 0.5) ..
        controls +(0,0.4) and +(0,0.4) .. (1,0.5)
        -- (1,0);
        \FillGrey (1.6,0)--(1.6, 0.5) ..
        controls +(0,0.4) and +(0,0.4) .. (1,0.5)
        -- (1,0);
        \ThinLine (0.3,0)
        -- (0.3,1);
        \FillGrey (-0.3,0)--(0.3,0)--(0.3,1)--(-0.3,1); 
        \end{scope}
        %\HelpLines
    \end{tikzpicture}}}$ represents the bimodule morphism $_{\cM}\cM_{\cN}\ni x\mapsto \delta^{1/2}x\otimes_{\cN}1\in _{\cM}\cM\otimes_{\cN}\cM_{\cN}$, and 
    $\vcenter{\hbox{\begin{tikzpicture}[scale=0.6, rotate = 180]
        \begin{scope}% Fourier transform of 2-box
        \ThinLine (1.6,0)--(1.6, 0.5) ..
        controls +(0,0.4) and +(0,0.4) .. (1,0.5)
        -- (1,0);
        \FillGrey (1.6,0)--(1.6, 0.5) ..
        controls +(0,0.4) and +(0,0.4) .. (1,0.5)
        -- (1,0);
        \ThinLine (0.3,0)
        -- (0.3,1);
        \FillGrey (-0.3,0)--(0.3,0)--(0.3,1)--(-0.3,1); 
        \end{scope}
        %\HelpLines
    \end{tikzpicture}}}$ the bimodule morphism $_{\cN}\cM\otimes_{\cN}\cM_{\cM} \ni x\otimes_{\cN}y\to \delta^{1/2} E_{\cN}(x)y\in _{\cN}\cM_{\cM}$. 
    Therefore with Proposition \ref{prop::bimodule of a condiitonal expectation}, we obtain:
    \begin{align*}
        \vcenter{\hbox{\begin{tikzpicture}[scale=0.6]
            \begin{scope}% Fourier transform of 2-box
            \ThinLine (1.6,0)--(1.6, 0.5) ..
            controls +(0,0.4) and +(0,0.4) .. (1,0.5)
            -- (1,0);
            \FillGrey (1.6,0)--(1.6, 0.5) ..
            controls +(0,0.4) and +(0,0.4) .. (1,0.5)
            -- (1,0);
            \ThinLine (0.3,0)
            -- (0.3,1);
            \FillGrey (-0.3,0)--(0.3,0)--(0.3,1)--(-0.3,1); 
            \end{scope}
            %\HelpLines
        \end{tikzpicture}}}=\delta^{1/2}u_{E_{\cN},\tau_{\cM}},\quad 
        \vcenter{\hbox{\begin{tikzpicture}[scale=0.6, rotate = 180]
            \begin{scope}% Fourier transform of 2-box
            \ThinLine (1.6,0)--(1.6, 0.5) ..
            controls +(0,0.4) and +(0,0.4) .. (1,0.5)
            -- (1,0);
            \FillGrey (1.6,0)--(1.6, 0.5) ..
            controls +(0,0.4) and +(0,0.4) .. (1,0.5)
            -- (1,0);
            \ThinLine (0.3,0)
            -- (0.3,1);
            \FillGrey (-0.3,0)--(0.3,0)--(0.3,1)--(-0.3,1); 
            \end{scope}
            %\HelpLines
        \end{tikzpicture}}}=\delta^{1/2}v^*_{E_{\cN},\tau_{\cM}}.
    \end{align*}
    Consequently by Proposition \ref{prop:: The derivative implements Phi}: 
    \begin{align*}
        \delta^{-1}\vcenter{\hbox{\begin{tikzpicture}[scale=0.6]
            \begin{scope}% Fourier transform of 2-box
            \ThinLine (1.6,-0.5)--(1.6, 1) ..
            controls +(0,0.4) and +(0,0.4) .. (1,1)
            -- (1,-0.5);
            \FillGrey (1.6,-0.5)--(1.6, 1) ..
            controls +(0,0.4) and +(0,0.4) .. (1,1)
            -- (1,-0.5);
            \ThinLine (-0.3,1.5)--(-0.3,0) ..
            controls +(0,-0.4) and +(0,-0.4) .. (0.3,0)
            -- (0.3,1.5);
            \FillGrey (-0.3,1.5)--(-0.3,0) ..
            controls +(0,-0.4) and +(0,-0.4) .. (0.3,0)
            -- (0.3,1.5);
            \draw[fill=white] (0,0) rectangle (1.3,1);
            \node[scale = 0.9] at (0.65,0.5) {$h_{\Phi}$};
            \end{scope}
            %\HelpLines
        \end{tikzpicture}}} = v^*_{E_{\cN},\tau_{\cM}}h_{\Phi}u_{E_{\cN},\tau_\cM}=V^{\tau_{\cM}}_{\Phi}=
        \vcenter{\hbox{\begin{tikzpicture}
            \begin{scope}[scale=0.6]% 2-box
            \FillGrey (0.3,-0.5) -- (1,-0.5)
            -- (1,1.5) -- (0.3,1.5);
            \VeryThinLine (0.3,-0.5) -- (0.3,1.5);
            \VeryThinLine (1,-0.5) -- (1,1.5);
            \draw[fill=white] (0,0) rectangle (1.3,1);
            \node at (0.65,0.5) {$\Phi$};
            \end{scope}
        \end{tikzpicture}}}.
    \end{align*}
    Therefore by planar isotopy, we have
    \begin{align*}
        \vcenter{\hbox{\begin{tikzpicture}
            \begin{scope}[scale=0.5]% 2-box
            \FillGrey (-0.3,-0.5) -- (0.3,-0.5)
            -- (0.3,1.5) -- (-0.3,1.5);
            \FillGrey (1,-0.5) -- (1.6,-0.5)
            -- (1.6,1.5) -- (1,1.5);
            \VeryThinLine (0.3,-0.5) -- (0.3,1.5);
            \VeryThinLine (1,-0.5) -- (1,1.5);
            \draw[fill=white] (0,0) rectangle (1.3,1);
            \node at (0.65,0.5) {$h_{\Phi}$};
            \end{scope}
          \end{tikzpicture}}} = 
          \delta \vcenter{\hbox{\begin{tikzpicture}[scale=0.6]
            \begin{scope}% Fourier transform of 2-box
            \ThinLine (2.2,2) -- (2.2,-0.5) .. controls +(0,-0.4) and +(0,-0.4) .. (1.6,-0.5)--(1.6, 1) ..
            controls +(0,0.4) and +(0,0.4) .. (1,1)
            -- (1,-1);
            \FillGrey (1,-1) -- (2.8,-1) --(2.8,2) -- (2.2,2) -- (2.2,-0.5) .. controls +(0,-0.4) and +(0,-0.4) .. (1.6,-0.5)--(1.6, 1) ..
            controls +(0,0.4) and +(0,0.4) .. (1,1)
            -- (1,-1);
            \ThinLine (-0.9,-1)--(-0.9,1.5).. controls +(0,0.4) and +(0,0.4).. (-0.3,1.5)--(-0.3,0) ..
            controls +(0,-0.4) and +(0,-0.4) .. (0.3,0)
            -- (0.3,2);
            \FillGrey (0.3,2) -- (-1.5,2) -- (-1.5,-1) -- (-0.9,-1)--(-0.9,1.5).. controls +(0,0.4) and +(0,0.4).. (-0.3,1.5)--(-0.3,0) ..
            controls +(0,-0.4) and +(0,-0.4) .. (0.3,0)
            -- (0.3,2);
            \draw[fill=white] (0,0) rectangle (1.3,1);
            \node[scale = 0.9] at (0.65,0.5) {$h_{\Phi}$};
            \end{scope}
            %\HelpLines
        \end{tikzpicture}}} = 
        \delta\vcenter{\hbox{\begin{tikzpicture}[scale=0.5]
            \begin{scope}% Fourier transform of 2-box
                \FillGrey (-0.6,-0.5) -- (1.9,-0.5)
            -- (1.9,1.5) -- (-0.6,1.5);
            \ThinLine (0.3,-0.5)--(0.3, 1) ..
            controls +(0,0.4) and +(0,0.4) .. (-0.3,1)
            -- (-0.3,-0.5);
            \FillWhite (0.3,-0.5)--(0.3, 1) ..
            controls +(0,0.4) and +(0,0.4) .. (-0.3,1)
            -- (-0.3,-0.5)--(0.3,-0.5);
            \ThinLine (1,1.5)--(1,0) ..
            controls +(0,-0.4) and +(0,-0.4) .. (1.6,0)
            -- (1.6,1.5);
            \FillWhite (1,1.5)--(1,0) ..
            controls +(0,-0.4) and +(0,-0.4) .. (1.6,0)
            -- (1.6,1.5)-- (1,1.5);
            \draw[fill=white] (0,0) rectangle (1.3,1);
            \node at (0.65,0.5) {$\Phi$};
            \end{scope}
            %\HelpLines
        \end{tikzpicture}}} = \delta\widehat{\Phi}.
    \end{align*}
    Algebraically, we check that for all $x\in\cM$, 
    \begin{align*}
        \delta v^*_{E_{\cN},\tau_{\cM}} \widehat{\Phi} u_{E_{\cN},\tau_{\cM}} (x\Omega_{\cM}) &= \sum^n_{j=1}v^*_{E_{\cN},\tau_{\cM}} \big(x \eta_j \Omega_{\cM}\otimes_{\cN} \Phi(\eta^*_j) \Omega_{\cM}\big)\\
        &= \sum^n_{j=1} \Phi(E_{\cN}(x\eta_j)\eta^*_j)\Omega_{\cM} = \Phi(x)\Omega_{\cM}.
    \end{align*}
    Hence by uniqueness of $h_{\Phi}$ as in Proposition \ref{prop:: The derivative implements Phi} the result follows. 
\end{proof}

Suppose $\Phi_1\preccurlyeq \Phi_2$ and $\Psi_1\preccurlyeq \Psi_2$ are normal completely positive and we are allowed to compose $\Phi_1,\Phi_2$ with $\Psi_1,\Psi_2$. 
It follows that $\Phi_1\Phi_2\preccurlyeq \Psi_1\Psi_2$, hence $h_{\Phi_1\Phi_2,\Psi_1\Psi_2}$ exists. 
In the following we will prove a formula expressing $h_{\Phi_1\Phi_2,\Psi_1\Psi_2}$ in terms of $h_{\Phi_1,\Phi_2}$ and $h_{\Psi_1,\Psi_2}$. 

First we briefly recall from \cite{Takesaki2003} relative tensor product of bimodules over von Neumann algebras.
Suppose $H_{\mathcal{B}}$ and $_{\mathcal{B}}K$are $\mathcal{B}$-modules, and $\varphi$ is a  faithful normal state on $\mathcal{B}$. A vector $\xi\in H$ is called $\varphi$-bounded if the densely defined operator 
\begin{align*}
   L_{\varphi}(\xi):L^2(\mathcal{B},\varphi)\ni\Omega_{\varphi}b\mapsto \xi\cdot b\in H 
\end{align*}
admits a bounded extension, which is still denoted as $L_{\varphi}(\xi)$. 
Denote the dense suspace of $\varphi$-bounded vectors in $H_{\mathcal{B}}$ as $\mathfrak{D}(H,\varphi)$. 
Define a positive bilinear form on $\mathfrak{D}(H,\varphi)\otimes K$ as
\begin{align*}
    \left\langle \xi_1\otimes\eta_1,\xi_2\otimes\eta_2\right\rangle _{\mathcal{B},\varphi}=\left\langle \pi_K(L^*_{\varphi}(\xi_2)L_{\varphi}(\xi_1))\eta_1,\eta_2\right\rangle _K.
\end{align*}
The \textbf{relative tensor product} of $H$ and $K$ over $\mathcal{B}$ with respect to $\varphi$, denoted as $H\otimes_\varphi K$, will be the completion of $\mathfrak{D}(H,\varphi)\otimes K/\ker\left\langle\cdot,\cdot\right\rangle _{\mathcal{B},\varphi}$ with respect to the norm induced by the bilinear form. 
The equivalence class of a vector $\xi\otimes\eta\in \mathfrak{D}(H,\varphi)\otimes K$ will be denoted as $\xi\otimes_\varphi \eta$. 
% Recall that for $b\in\mathfrak{D}(\sigma^{\varphi}_{i/2})$, we have $\xi b\otimes_{\varphi}\eta=\xi\otimes_{\varphi}\sigma^{\varphi}_{i/2}(b)\eta$. 
In case where $H$ is a $\mathcal{A}$-$\mathcal{B}$-bimodule and $K$ is a $\mathcal{B}$-$\mathcal{C}$-bimodule, the left (right) actions of $\mathcal{A}$ ($\mathcal{C}$) descend to $H\otimes_\varphi K$, making it a $\mathcal{A}$-$\mathcal{C}$-bimodule. 
The map $\iota: _\mathcal{B}K\rightarrow _\mathcal{B}L^2(\mathcal{B})\otimes_\varphi K$ defined as $\iota(b\cdot\eta)=b\Omega_\varphi\otimes_\varphi\eta$ extends to a unitary intertwiner of left $\mathcal{B}$-representations.\\
For $x\in \text{End}(_\mathcal{A}H_\mathcal{B})$ and $y\in \text{End}(_\mathcal{B}K_\mathcal{C})$, the operator $x\otimes_{\varphi} y\in \text{End}(_\mathcal{A}H\otimes_{\varphi} K_\mathcal{C})$ is defined as 
$$(x\otimes_{\varphi} y)(\xi\otimes_\varphi \eta)=x\xi\otimes_\varphi y\eta,$$
and is called the tensor product of $x$ and $y$.

\begin{lemma}\label{lemma:minimal}
    Suppose $ \Psi_1:\mathcal{B}\rightarrow\mathcal{C}$ and $\Psi_2: \mathcal{A}\to\mathcal{B}$ are normal completely positive. Fix $\varphi$ to be a faithful normal state on $\mathcal{B}$. Then there is an isometry $\mathcal{Y}\in \Hom (_{\mathcal{A}}\mathcal{H}^{\Psi_1\Psi_2}_{\mathcal{C}},_{\mathcal{A}}\mathcal{H}^{\Psi_2}\otimes_{\varphi}\mathcal{H}^{\Psi_1}_{\mathcal{C}})$ satisfying $\mathcal{Y}(\Omega_{\Psi_1\Psi_2,\phi})=\Omega_{\Psi_2,\varphi}\otimes_{\varphi}\Omega_{\Psi_1,\phi}$. 
    Consequently 
    \begin{align}
        \mathcal{Y}v_{\Psi_1\Psi_2}=(v_{\Psi_2}\otimes_{\varphi} id_{\mathcal{H}^{\Psi_1}})v_{\Psi_1}.
    \end{align}
\end{lemma}
\begin{proof}
    We first show that the assignment $\mathcal{Y}: a\Omega_{\Psi_1\Psi_2}c\mapsto a\Omega_{\Psi_2}\otimes_{\varphi}\Omega_{\Psi_1} c$ preserves inner product.
    Indeed, for any $a_1,a_2\in\mathcal{A}$ and $c_1,c_2\in\mathcal{C}$:
    \begin{align*}
        &\left\langle a_1\Omega_{\Psi_2}\otimes_{\varphi}\Omega_{\Psi_1} c_1,a_2\Omega_{\Psi_2}\otimes_{\varphi}\Omega_{\Psi_1} c_2\right\rangle\\
        &=\left\langle L^*_{\varphi}(\Omega_{\Psi_2})a^*_2a_1L_{\varphi}(\Omega_{\Psi_2})\Omega_{\Psi_1}c_1, \Omega_{\Psi_1}c_2\right\rangle\\
        &=\left\langle \Psi_2(a^*_2a_1)\Omega_{\Psi_1}c_1, \Omega_{\Psi_1}c_2\right\rangle\\
        &=\left\langle \Psi_1(\Psi_2(a^*_2a_1))\Omega_{\phi}c_1,\Omega_{\phi}c_2\right\rangle\\
        &=\left\langle a_1\Omega_{\Psi_1\Psi_2} c_1,a_2\Omega_{\Psi_1\Psi_2}c_2\right\rangle.
    \end{align*}
    Since $\Omega_{\Phi}$ is cyclic in $\mathcal{H}^{\Psi_1\Psi_2}$, $\mathcal{Y}$ extends linearly to an isometric bimodule map from $\mathcal{H}^{\Psi_1\Psi_2}$ into $\mathcal{H}^{\Psi_2}\otimes_{\varphi}\mathcal{H}^{\Psi_1}$.
    By the definition of $v_{\Psi}$, we have the equation $\mathcal{Y}v_{\Psi_1\Psi_2}=(v_{\Psi_2}\otimes_{\varphi} id_{\mathcal{H}^{\Psi_1}})v_{\Psi_1}$. 
\end{proof}
\begin{proposition}\label{prop::convolution}
    Under the same assumption as in Lemma \ref{lemma:minimal}, suppose $\Phi_1\preccurlyeq \Psi_1$ and $\Phi_2\preccurlyeq \Psi_2$ are normal completely positive. 
    Then $\Phi_1\Phi_2\preccurlyeq\Psi_1\Psi_2$ and
    \begin{align*}
        h_{\Phi_1\Phi_2,\Psi_1\Psi_2}=\mathcal{Y}^*(h_{\Phi_2,\Psi_2}\otimes_{\varphi} h_{\Phi_1,\Psi_1})\mathcal{Y}.
    \end{align*}
\end{proposition}
\begin{proof}
    For any $ a\in\mathcal{A}$, by Lemma \ref{lemma:minimal}, 
    \begin{align*}
        &v_{\Psi_1\Psi_2}^*\big(\mathcal{Y}^*(h_{\Phi_2,\Psi_2}\otimes_{\varphi} h_{\Phi_1,\Psi_1})\mathcal{Y}\big)\pi_{\Psi_1\Psi_2}(a)v_{\Psi_1\Psi_2}\\
        % &=v_{\Psi_1\Psi_2}^*w^*(h_{\Phi_2,\Psi_2}\otimes_{\varphi} h_{\Phi_1,\Psi_1})\pi_{\Psi_2}(a)wv_{\Psi_1\Psi_2}\\
        &=v^*_{\Psi_1}(v^*_{\Psi_2}\otimes_{\varphi} id_{\mathcal{H}^{\Psi_1}})\pi_{\Psi_2}(a)\left(h_{\Phi_2,\Psi_2}\otimes_{\varphi} h_{\Phi_1,\Psi_1}\right)(v_{\Psi_2}\otimes_{\varphi} id_{\mathcal{H}^{\Psi_1}})v_{\Psi_1}\\
        &=v^*_{\Psi_1}(v^*_{\Psi_2}h_{\Phi_2,\Psi_2}^{1/2}\otimes_{\varphi} h_{\Phi,\Psi}^{1/2})\pi_{\Psi_2}(a)(h_{\Phi_2,\Psi_2}^{1/2}v_{\Psi_2}\otimes_{\varphi} h_{\Phi,\Psi}^{1/2})v_{\Psi_1}\\
        &=v^*_{\Psi_1}h_{\Phi,\Psi}^{1/2}\pi_{\Psi_1}(\Phi_2(a))h_{\Phi,\Psi}^{1/2}v_{\Psi_1}\\
        &=\Phi_1(\Phi_2(a)).
    \end{align*}
    Hence by the uniqueness of derivative in Proposition \ref{lemma:RN}, we have $h_{\Phi_1\Phi_2,\Psi_1\Psi_2}=\mathcal{Y}^*(h_{\Phi_2,\Psi_2}\otimes_{\varphi} h_{\Phi_1,\Psi_1})\mathcal{Y}$.
\end{proof}
    We explain the role of isometry $w$ appeared in Proposition \ref{prop::convolution} in connection with convolution studied in planar algebra settings. 
    Suppose now $\mathcal{A}=\mathcal{B}=\mathcal{C}=\cM$ is a II$_1$ factor and $\Psi_1=\Psi_2 =E_{\cN}$ is the trace-preserving conditional expectation down to a finite index subfactor. 
    Then $w\in \Hom (_{\cM}\cM\otimes_{\cN}\cM_{\cM},_{\cM}\cM\otimes_{\cN}\cM\otimes_{\cN}\cM_{\cM})$ is represented by the following diagram in $\mathscr{P}^{\cN\subset\cM}$:
    \begin{align*}
        \mathcal{Y} = \delta^{-1/2}\cdot \vcenter{\hbox{\begin{tikzpicture}[scale=0.6]
            \begin{scope}
            \ThinLine (0.6,0)--(0.6, 0.5) ..
            controls +(0,0.4) and +(0,0.4) .. (0,0.5)
            -- (0,0);
            \FillGrey (0.6,0)--(0.6, 0.5) ..
            controls +(0,0.4) and +(0,0.4) .. (0,0.5)
            -- (0,0);
            \ThinLine (-0.6,0) .. controls +(0,0.5) and +(0,-0.5) .. (0,1.5);
            \ThinLine (1.2,0) .. controls +(0,0.5) and +(0,-0.5) .. (0.6,1.5);
            \FillGrey (-1,0) -- (-0.6,0) .. controls +(0,0.5) and +(0,-0.5) .. (0,1.5) 
            -- (-1,1.5) -- (-1,0); 
            \FillGrey (1.6,0) -- (1.2,0) .. controls +(0,0.5) and +(0,-0.5) .. (0.6,1.5) 
            -- (1.6,1.5) -- (1.6,0); 
            \end{scope}
            %\HelpLines
        \end{tikzpicture}}}
    \end{align*}
    As proved in Section 4 for $\Phi_1,\Phi_2\in \mathbf{CP}_{\cN}(\cM)$ we have $h_{\Phi_i} = \delta \widehat{\Phi}_i$. 
    Now Proposition \ref{prop::convolution} translates to the following equation of two boxes:
    \begin{align*}
        \widehat{\Phi_2\Phi_1} = \vcenter{\hbox{\begin{tikzpicture}
            \begin{scope}[scale=0.6]% Convolution
        \FillWhite (-1.3,1.5) -- (-1.3, -0.5) -- (1.3, -0.5) -- (1.3, 1.5) -- (-1.3,1.5);
        \ThinLine (1.3, -0.5) -- (1.3, 1.5);
        \ThinLine (-1.3,-0.5) -- (-1.3,1.5);
        \ThinLine (0.6,1) .. controls +(0,0.4) and +(0,0.4) .. (-0.6,1);
        \ThinLine (0.6,0) .. controls +(0,-0.4) and +(0,-0.4) .. (-0.6,0);
        \FillGrey (0.6,1) .. controls +(0,0.4) and +(0,0.4) .. (-0.6,1)
        -- (-0.6,0) -- (-0.6,0) .. controls +(0,-0.4) and +(0,-0.4) .. (0.6,0)
        -- (0.6,1);
        \FillGrey (1.3,-0.5) -- (1.9,-0.5) -- (1.9,1.5) -- (1.3,1.5);
        \FillGrey (-1.9,-0.5) -- (-1.3,-0.5) -- (-1.3,1.5) -- (-1.9,1.5) -- (-1.9,-0.5);
        \draw[fill=white] (0.3,0) rectangle (1.6,1);
        \node at (0.95,0.5) {$\widehat{\Phi}_2$};
        \draw[fill=white] (-1.6,0) rectangle (-0.3,1);
        \node at (-0.95,0.5) {$\widehat{\Phi}_1$};
    \end{scope}
    \end{tikzpicture}}}.
    \end{align*}
    This is nothing but the fact that Fourier transform intertwines composition and convolution (c.f. Equation (2) \cite{JJLRW2020}). 
% \begin{theorem}
% Suppose $\mathcal{N}\subset \mathcal{M}$ is a II$_1$ subfactor of finite index and $\Phi, \Psi$ are $\cN$-$\cN$-bimodule quantum channels.
% Then 
% \begin{align*}
% H(\Phi|\Psi)=[\cM:\cN] S(h_{\Phi^{\Delta}}, h_{\Psi^{\Delta}}).
% \end{align*}
% \end{theorem}
% \begin{proof}

% \end{proof}

\section{Araki Relative Entropy for Completely Positive Maps}
In \cite{Araki1977}, Araki extends the relative entropy between density operators to positive linear functionals on arbitrary von Neumann algebras. 
Based on this notion together with Connes' correspondences (bimodules) we define and study the relative entropy between completely positive maps. 

Let us recall Araki's definition of relative entropy. 
Given normal normal positve linear functionals $\rho,\sigma$ on a von Neumman algebra $\mathcal{R}$ in its standard form, we use $p_{\sigma}$ and $p'_{\sigma}$ to represent the supports of $\sigma$ in $\mathcal{R}$ and $\mathcal{R}'$. 
Let $\xi_{\rho},\xi_{\sigma}$ be their unique representatives in the natural positive cone. 
When $p_{\rho}\leq p_{\sigma}$, the densely defined conjugate linear operator
\begin{align*}
    S_{\rho,\sigma}: x\xi_{\sigma}+\eta\mapsto p_{\sigma}x^*\xi_{\rho}, \quad  x\in\mathcal{R},\quad p'_{\sigma}\eta = 0.
\end{align*}
is closable (with closure still denoted as $S_{\rho,\sigma}$) and the relative modular operator is the positive selfadjoint operator $\Delta_{\varphi,\psi}=S^*_{\rho,\sigma}S_{\rho,\sigma}$. 
The relative entropy between $\rho,\sigma$ is defined as
\begin{align*}
    S(\rho,\sigma)=\left\langle \log \Delta_{\rho,\sigma}\xi_{\rho},\xi_{\rho} \right\rangle.
\end{align*}

Let $\mathcal{A},\mathcal{B}$ be von Neumann algebras and $\Psi\in\mathbf{CP}(\mathcal{A},\mathcal{B})$. 
We fix a faithful normal state $\varphi$ on $\mathcal{B}$ and consider the $\mathcal{A}$-$\mathcal{B}$ bimodule $\mathcal{H}^\Phi$. 
Denote the normal faitfhul positive linear functional on $\End(_{\mathcal{A}}\mathcal{H}^{\Psi}_{\mathcal{B}})$ implemented by $\Omega_{\Psi,\varphi}$ as $\omega_{\Psi,\varphi}$. 
%Fix this!%
Now let $\End(_{\mathcal{A}}\mathcal{H}^{\Phi}_{\mathcal{B}})$ be in its standard form and $\xi_{\Psi,\varphi}$ be the unique representative of $\omega_{\Psi,\varphi}$ in the natural positive cone. Denote the modular conjugation associated to $\omega_{\Psi,\varphi}$ as $J_{\Psi}$. 
For every $\Phi\in \mathbf{CP}(\mathcal{A},\mathcal{B})$, define a positive normal functional on $\End(_{\mathcal{A}}\mathcal{H}^{\Phi}_{\mathcal{B}})$ as
\begin{equation}\label{eqn:: channel-state correspondence}
    \omega_{\Psi,\varphi}(\Phi) (x)  =\langle x\xi_{\Psi,\varphi}, J_{\Psi}h_{\Phi,\Psi}\xi_{\Psi,\varphi}\rangle,\quad x\in \End(_{\mathcal{A}}\mathcal{H}^{\Phi}_{\mathcal{B}}).
\end{equation}
Notice that if $\Phi$ is unital and $\varphi$ is a state, then so is $\omega_{\Psi,\varphi}(\Psi)$. 
By definition we have $\omega_{\Psi,\varphi} = \omega_{\Psi,\varphi}(\Psi)$. 
\begin{definition}\label{def:: Relative entropy for CP maps}
Suppose $\Psi,\Phi:\mathcal{A}\to\mathcal{B}$ are completely positive maps and $\Phi \preccurlyeq\Psi$.
Let $\varphi$ be a faithful normal state on $\mathcal{B}$.
Then the relative entropy $S_\varphi(\Phi, \Psi)$ of $\Phi,\Psi$ with respect to $\varphi$ is defined to be:
    \begin{align}
        S_\varphi(\Phi,\Psi)
        =S( \omega_{\Psi,\varphi}(\Phi) ,\omega_{\Psi,\varphi}(\Psi)).
    \end{align}
\end{definition}
Alternatively $S_\varphi(\Phi,\Psi)$ can be defined using the corresponding positive functionals on $J_{\Psi}\End(_{\mathcal{A}}\mathcal{H}^{\Psi}_{\mathcal{B}})J_{\Psi} = \End(_{\mathcal{A}}\mathcal{H}^{\Psi}_{\mathcal{B}})^{\textbf{op}}$. 
Using the modular conjugation $j: \End(_{\mathcal{A}}\mathcal{H}^{\Psi}_{\mathcal{B}})^{\textbf{op}}\ni x\mapsto J_{\Psi}x^*J_{\Psi}\in \End(_{\mathcal{A}}\mathcal{H}^{\Psi}_{\mathcal{B}})$, we define the state on $\End(_{\mathcal{A}}\mathcal{H}^{\Psi}_{\mathcal{B}})^{\textbf{op}}$ as
\begin{align*}
    \omega'_{\Psi,\phi}(\Phi) := \omega_{\Psi,\varphi}(\Phi)\circ j.
\end{align*}
Then $ \omega'_{\Phi,\Psi}(\Phi)$ is implemented by the vector $h^{1/2}_{\Phi,\Psi}\xi_{\Psi,\varphi}$. 
Therefore by the fact that $j$ is an anti-isomorphism, we arrived at the following proposition.

\begin{proposition}\label{prop:: alternative definition of relative entropy}
    Under the assumption of Definition \ref{def:: Relative entropy for CP maps}, we have
    \begin{align*}
        S_{\varphi}(\Phi,\Psi)=S(\omega'_{\Psi,\varphi}(\Phi),\omega'_{\Psi,\varphi}(\Psi)).
    \end{align*}
\end{proposition}
\begin{lemma}\label{lemma::IndependeceOfReference}
    Suppose $\Psi,\Phi,\mathcal{E}: \mathcal{A}\to\mathcal{B}$ are completely positive and $\Phi \preccurlyeq\Psi\preccurlyeq\mathcal{E}$. 
    Let $\varphi$ be a faithful normal state on $\mathcal{B}$, then 
    \begin{align}
        S_{\varphi}(\Phi,\Psi)=S(\omega_{\mathcal{E},\varphi}(\Phi),\omega_{\mathcal{E},\varphi}(\Psi)).
    \end{align}
\end{lemma}
\begin{proof}
    Let $(_{\mathcal{A}}\mathcal{H}^{\mathcal{E}}_{\mathcal{B}},\Omega_{\mathcal{E}})$ be the dilation of $\mathcal{E}$ with respect to $\varphi$. 
    By definition of the derivative between completely positive maps the assignment $a\Omega_{\Psi}b\mapsto ah^{1/2}_{\Psi,\mathcal{E}}\Omega_{\mathcal{E}}b$ for all $a\in\mathcal{A}$ and $b\in\mathcal{B}$ extends to an isometric bimodule intertwiner. 
   
    Let $p_0$ be the support of $h^{1/2}_{\Psi,\mathcal{E}}$. 
    In the rest of the proof, we identify $(\mathcal{H}^{\Psi},\Omega_{\Psi})$ with $(p_0\mathcal{H}^{\mathcal{E}},h^{1/2}_{\Psi,\mathcal{E}}\Omega_{\mathcal{E}})$ as $\mathcal{A}$-$\mathcal{B}$-bimodules. 
    Consequently $\End(_{\mathcal{A}}\mathcal{H}^{\Psi}_{\mathcal{B}})$ is identified with $p_0\End(_{\mathcal{A}}\mathcal{H}^{\mathcal{E}}_{\mathcal{B}})p_0$, 
    and $J_{\omega_\Psi}\End(_{\mathcal{A}}\mathcal{H}^{\Psi}_{\mathcal{B}})J_{\omega_\Psi}$ is identified with $J_{\omega_\mathcal{E}}\End(_{\mathcal{A}}\mathcal{H}^{\mathcal{E}}_{\mathcal{B}})J_{\omega_\mathcal{E}}p_0$. 

    Now we find that $\omega'_{\Psi}$ is implemented by the vector $h^{1/2}_{\Psi,\mathcal{E}}\Omega_{\mathcal{E}}$. 
    By uniqueness of the derivative $h_{\Phi,\Psi}=h^{-1/2}_{\Psi,\mathcal{E}}h_{\Phi,\mathcal{E}}h^{-1/2}_{\Psi,\mathcal{E}}$ as an operator on $p_0\mathcal{H}^{\mathcal{E}}_0$, so $\omega'_{\Psi}(\Phi)$ is implemented by $h^{1/2}_{\Phi,\mathcal{E}}\Omega_{\mathcal{E}}$. 
    Therefore by Corollary \ref{prop:: alternative definition of relative entropy}: 
    \begin{align*}
        S_{\varphi}(\Phi,\Psi) &= S(\omega'_{\Psi,\varphi}(\Phi),\omega'_{\Psi,\varphi})\\
        &= S(\omega'_{\mathcal{E},\varphi}(\Phi),\omega'_{\mathcal{E},\varphi}(\Psi))\\
        &= S(\omega_{\mathcal{E},\varphi}(\Phi),\omega_{\mathcal{E},\varphi}(\Psi)), 
    \end{align*}
    completing the proof.
\end{proof}
We now interpret the upper bound in Equation \eqref{eqn::upper bound for H} as a special case of the quantity $S_{\varphi}(\Phi,\Psi)$. 
\begin{theorem}\label{thm::PimsnerPopaEntropy}
    Let $\mathcal{N}\subset \mathcal{M}$ is a finite inclusion of finite von Neumann algebras. Suppose $\Phi,\Psi \in \mathbf{CP}_{\cN}(\cM)$ and $\Phi\preccurlyeq \Psi$. 
    Then 
    \begin{align*}
        \delta D_{\tau_{\cM_2}}(\Delta^{1/2}\widehat{\Phi}\Delta^{1/2}\| \Delta^{1/2}\widehat{\Psi}\Delta^{1/2}) = S_{\tau_{\mathcal{M}}}(\Phi,\Psi).
    \end{align*}
\end{theorem}
\begin{proof}
    By Proposition \ref{lemma::IndependeceOfReference} we have
    \begin{align*}
        S_{\tau_{\mathcal{M}}}(\Phi,\Psi)=S(\omega_{E_{\cN},\tau_{\cM}}(\Phi),\omega_{E_{\cN},\tau_{\cM}}(\Psi)).
    \end{align*}
    For any $\Phi_0\in \mathbf{CP}_{\cN}(\cM)$, by Theorem \ref{thm:: Fourier multiplier is derivative} $\delta\widehat{\Phi}_0 = h_{\Phi_0,E_{\cN}}$. 
    Then Proposition \ref{prop:: characterize the state omega} implies
    \begin{align*}
         \delta \omega_{E_{\cN},\tau_{\cM}}(\widehat{\Phi}_0) = \omega_{E_{\cN},\tau_{\cM}}(h_{\Phi_0,E_{\cN}})=\tau_{\cM}(\Phi_0(1))=\delta\tau_{\cM_2}(\Delta\widehat{\Phi}_0),
    \end{align*}
    so we have $\omega_{E_{\cN},\tau_{\cM}} = \omega_{\cN}$ and $\Delta$ is the density operator of $\omega_{E_{\cN},\tau_{\cM}}$ with respect to $\tau_{\cM_2}$. 
    The density operators for $\omega_{E_{\cN},\tau_{\cM}}(\Phi)$ and $\omega_{E_{\cN},\tau_{\cM}}(\Psi)$ with respect to the trace can then be computed as $\Delta^{1/2}h_{\Phi}\Delta^{1/2}$ and $\Delta^{1/2}h_{\Psi}\Delta^{1/2}$ respectively. 
    Thus
    \begin{align*}
        S(\omega_{E_{\cN},\tau_{\cM}}(\Phi),\omega_{E_{\cN},\tau_{\cM}}(\Psi))&=D_{\tau_{\cM_2}}(\Delta^{1/2}h_{\Phi}\Delta^{1/2}\|\Delta^{1/2}h_{\Psi}\Delta^{1/2})\\
        &= \delta D_{\tau_{\cM_2}}(\Delta^{1/2}\widehat{\Phi}\Delta^{1/2}\| \Delta^{1/2}\widehat{\Psi}\Delta^{1/2}).
    \end{align*}
\end{proof}
In the end of this section, we discuss the monotonicity of $S_{\varphi}(\Phi,\Psi)$ under compositions with completely positive maps. 
\begin{theorem}[Right monotonicity]\label{thm::RightMonotonicity}
    Let $\Psi_1,\Phi_1\in \mathbf{CP}(\mathcal{B},\mathcal{C})$ with $\Phi_1 \preccurlyeq\Psi_1$. 
    Then for any faithful normal positive linear functional $\phi$ on $\mathcal{C}$ and $\Psi_2\in \Ch(\mathcal{A},\mathcal{B})$:
    \begin{align*}
        S_\phi(\Phi_1\Psi_2,\Psi_1\Psi_2) \leq S_\phi(\Phi_1,\Psi_1).
    \end{align*}
\end{theorem}
\begin{proof}
    We fix a faithful normal state $\varphi$ on $\mathcal{B}$ and adopt the setting of Proposition \ref{lemma:minimal}.
    Consider the map $\Gamma: \End(_\mathcal{B}\mathcal{H}^{\Psi_1}_\mathcal{C})\rightarrow \End(_\mathcal{A}\mathcal{H}^{\Psi_1\Psi_2}_\mathcal{C})$
    \begin{align}\label{eqn::Gammamap}
        \Gamma(y)=\mathcal{Y}^*(id_{\mathcal{H}^{\Psi_2}}\otimes_{\varphi} y)\mathcal{Y},
    \end{align}
    which is normal unital and completely positive. 
    Proposition \ref{prop::convolution} then reads (with $\Psi_2=\Phi_2$)
    \begin{align}
        \Gamma(h_{\Phi_1,\Psi_1})=h_{\Phi_1\Psi_2,\Psi_1\Psi_2}.
    \end{align}
    By defining properties of $w$ as in Lemma \ref{lemma:minimal} and $\Psi_2$ being unital one checks for any $y\in \End(_\mathcal{B}\mathcal{H}^{\Psi_1}_\mathcal{C})$, 
    \begin{align*}
        \omega_{\Psi_1\Psi_2}\circ\Gamma (y)=\left\langle \Omega_{\Psi_2}\otimes_{\varphi}y\Omega_{\Psi_1},\Omega_{\Psi_2}\otimes_{\varphi}\Omega_{\Psi_1}\right\rangle=\left\langle y\Omega_{\Psi_1},\Omega_{\Psi_1}\right\rangle=\omega_{\Psi_1}(y).
    \end{align*}
    Therefore $\omega_{\Psi_1\Psi_2}\circ\Gamma=\omega_{\Psi_1}$. 
    In particular, $\Gamma$ is faithful.
    
    Let $\Gamma^*_{\Psi_1\Psi_2}:=\Gamma^*_{\omega_{\Psi_1\Psi_2}}$ be the adjoint in the sense of Petz \cite{Petz1988}. 
    We have that $\Gamma^*_{\Psi_1\Psi_2}:\End(_\mathcal{A}\mathcal{H}^{\Psi_1\Psi_2}_\mathcal{C})\to \End(_\mathcal{B}\mathcal{H}^{\Psi_1}_\mathcal{C})$ and that
    \begin{align*}
        &\left\langle \Gamma^*_{\Psi_1\Psi_2}(x)\xi_{\Psi_1},J_{\Psi_1}y\xi_{\Psi_1}\right\rangle = \left\langle x\xi_{\Psi_1\Psi_2},J_{\Psi_1\Psi_2}\Gamma(y)\xi_{\Psi_1\Psi_2}\right\rangle,
    \end{align*}
    whenever  $x\in \End(_\mathcal{A}\mathcal{H}^{\Psi_1\Psi_2}_\mathcal{C}), y\in \End(_\mathcal{B}\mathcal{H}^{\Psi_1}_\mathcal{C})$.
    Taking $y=1$, we have $\omega_{\Psi_1}\circ\Gamma^*_{\Psi_1\Psi_2}=\omega_{\Psi_1\Psi_2}$. In addition to this, one has for any $x\in \End(_\mathcal{A}\mathcal{H}^{\Psi_1\Psi_2}_\mathcal{C})$:
    \begin{align*}
        \omega_{\Psi_1}(\Phi_1)\circ\Gamma^*_{\Psi_1\Psi_2}(x) &=\left\langle \Gamma^*_{\Psi_1\Psi_2}(x)\xi_{\Psi_1}, J_{\Psi_1}h_{\Phi_1,\Psi_1}\xi_{\Psi_1}\right\rangle\\
        &=\left\langle x\xi_{\Psi_1\Psi_2}, J_{\Psi_1\Psi_2}\Gamma(h_{\Phi_1,\Psi_1})\xi_{\Psi_1\Psi_2}\right\rangle\\
        &=\left\langle x\xi_{\Psi_1\Psi_2}, J_{\Psi_1\Psi_2}h_{\Phi_1\Psi_2,\Psi_1\Psi_2}\xi_{\Psi_1\Psi_2}\right\rangle\\
        &=\omega_{\Psi_1\Psi_2}(\Phi_1\Psi_2)(x).
    \end{align*}
    Therefore the monotonicity of $S_{\phi}$ follows from that of Araki's relative entropy: 
    \begin{align*}
        S_{\varphi}(\Phi_1\Psi_2,\Psi_1\Psi_2) &=S(\omega_{\Psi_1\Psi_2}(\Phi_1\Psi_2),\omega_{\Psi_1\Psi_2})\\
        &=S(\omega_{\Psi_1}(\Phi_1)\circ\Gamma^*_{\Psi_1\Psi_2},\omega_{\Psi_1}\circ\Gamma^*_{\Psi_1\Psi_2})\\
        &\leq S(\omega_{\Psi_1}(\Phi_1),\omega_{\Psi_1})\\
        &=S_{\varphi}(\Phi_1,\Psi_1).
    \end{align*}
    This completes the proof.
\end{proof}
\begin{theorem}[Left monotonicity]\label{thm::LeftMonotonicity}
    Let $\Psi_1\in \Ch(\mathcal{B},\mathcal{C})$ and $\Phi_2 \preccurlyeq\Psi_2: \mathcal{A}\to\mathcal{B}$ be normal completely positive. 
    Then for any faithful normal positive linear $\phi$ on $\mathcal{C}$:
    \begin{align*}
        S_\phi(\Psi_1\Phi_2,\Psi_1\Psi_2) \leq S_{\phi\circ\Psi_1}(\Phi_2,\Psi_2).
    \end{align*}
\end{theorem}
\begin{proof}
    The proof is similar in structure to the previous one. 
    We again adopt the setting of Proposition \ref{prop::convolution} and define the normal unital completely positive map $\Lambda: \End(_\mathcal{A}\mathcal{H}^{\Psi_2}_\mathcal{B})\rightarrow \End(_\mathcal{A}\mathcal{H}^{\Psi_1\Psi_2}_\mathcal{C})$ as
    \begin{align*}
\Lambda(y)=\mathcal{Y}^*(y\otimes_{\phi\circ\Psi_1}id_{\mathcal{H}^{\Psi_1}})\mathcal{Y}.
    \end{align*}
    In what follows, we will use the full notation $\omega_{\Psi,\varphi}$. 
    Now we check that for any $x\in \End(_\mathcal{A}\mathcal{H}^{\Psi_2}_\mathcal{B})$:
    \begin{align*}
        \omega_{\Psi_1\Psi_2,\phi}\circ\Lambda(x)
        &=\left\langle x\Omega_{\Psi_2,\phi\circ\Psi_1}\otimes_{\phi\circ\Psi_1}\Omega_{\Psi_1},\Omega_{\Psi_2,\phi\circ\Psi_1}\otimes_{\phi\circ\Psi_1}\Omega_{\Psi_1}\right\rangle\\
        &=\left\langle v^*_{\Psi_2,\phi\circ\Psi_1}yv_{\Psi_2,\phi\circ\Psi_1} \Omega_{\Psi_1},\Omega_{\Psi_1}\right\rangle\\
        &=\phi\circ\Psi_1\big(v^*_{\Psi_2, \phi\circ\Psi_1}yv_{\Psi_2,\phi\circ\Psi_1}\big)\\
        &=\omega_{\Psi_2,\phi\circ\Psi_1}(y).
    \end{align*}
    Therefore with $\Lambda^*_{\Psi_1\Psi_2,\phi}: \End(_\mathcal{A}\mathcal{H}^{\Psi_1\Psi_2}_\mathcal{C})\to \End(_\mathcal{A}\mathcal{H}^{\Psi_2}_\mathcal{B})$ being the Petz transpose of $\Lambda$ with respect to $\omega_{\Psi_1\Psi_2,\phi}$, we check that 
    \begin{align*}
        \omega_{\Psi_1\Psi_2,\phi}(\Psi_1\Phi_2)=&\omega_{\Psi_2,\phi\circ\Psi_1}(\Phi_2)\circ \Lambda^*_{\Psi_1\Psi_2,\phi},\\
        \omega_{\Psi_1\Psi_2,\phi}=&\omega_{\Psi_2,\phi\circ\Psi_1}\circ \Lambda^*_{\Psi_1\Psi_2,\phi}.
    \end{align*}
    The conclusion follows again from the monotonicity of Araki's relative entropy.
\end{proof}
\begin{corollary}[Convexity]\label{corollary:: convexity}
    Let $\{\Phi_i\}^n_{i=1}$ and $\{\Psi_i\}^n_{i=1}$ be two families of normal completely positive maps from $\mathcal{A}$ to $\mathcal{B}$. 
    Then for any faithful positive linear functional $\varphi$ on $\mathcal{B}$ and any probability mass function $\{p_i\}^n_{i=1}$, we have
    \begin{align*}
        S_{\varphi}\big(\sum^n_{i=1}p_i\Phi_i,\sum^n_{i=1}p_i\Psi_i\big)\leq \sum^n_{i=1} p_i S_{\varphi}(\Phi_i,\Psi_i). 
    \end{align*}
\end{corollary}
\begin{proof}
    Without the loss of generality, we assume $p_i\neq 0$ for all $i$. 
    Define $\mathcal{E}_1\in \mathbf{UCP}(\mathcal{A},\mathcal{A}^{\oplus n})$, and $\mathcal{E}_2\in \mathbf{UCP}(\mathcal{B}^{\oplus n},\mathcal{B})$ as
    \begin{align*}
        &\mathcal{E}_1 (a) = (a)^n_{i=1},\quad  a\in\mathcal{A};\\
        &\mathcal{E}_2 \big((b_i)^n_{i=1}\big) = \sum^n_{i=1} p_ib_i,\quad  (b_i)^n_{i=1}\in \mathcal{B}^{\oplus n}.
    \end{align*}
    Meanwhile define 
    $\mathbf{\Phi},\mathbf{\Psi}\in \mathbf{CP}(\mathcal{A}^{\oplus n},\mathcal{B}^{\oplus n})$ as
    \begin{align*}
        &\mathbf{\Phi}\big((a_i)^n_{i=1}\big) = (\Phi_i(a_i))^n_{i=1},\quad (a_i)^n_{i=1}\in \mathcal{A}^{\oplus n};\\
        &\mathbf{\Psi}\big((a_i)^n_{i=1}\big) = (\Psi_i(a_i))^n_{i=1},\quad (a_i)^n_{i=1}\in \mathcal{A}^{\oplus n}.
    \end{align*}
    It is then straightforward to check that 
    \begin{align*}
        S_{\varphi}\big(\sum^n_{i=1}p_i\Phi_i,\sum^n_{i=1}p_i\Psi_i\big) = S_{\varphi} (\mathcal{E}_2\mathbf{\Phi}\mathcal{E}_1,\mathcal{E}_2\mathbf{\Psi}\mathcal{E}_1).
    \end{align*}
    By right monotonicity (Theorem \ref{thm::RightMonotonicity}), then by left monotonicity (Theorem \ref{thm::LeftMonotonicity}), we obtain
    \begin{align*}
        S_{\varphi} (\mathcal{E}_2\mathbf{\Phi}\mathcal{E}_1,\mathcal{E}_2\mathbf{\Psi}\mathcal{E}_1)\leq S_{\varphi \circ\mathcal{E}_2}(\mathbf{\Phi},\mathbf{\Psi}). 
    \end{align*}
    Since $\varphi \circ\mathcal{E}_2 = (p_i\varphi)^n_{i=1}$, the result follows. 
\end{proof}
% \begin{remark}
%     Using convexity we can obtain a simple proof of Theorem \ref{thm::popaentropycompletepositive}. 
%     First note that 
%     \begin{align*}
%         H(\Phi|\Psi) = \sup \sum^n_{i=1} p_i S(\varphi_i \circ\Phi,\varphi_i\circ \Psi), 
%     \end{align*}
%     where the sup is over all finite sets of normal states on $\cM$ and probabilities $\{p_i\}^n_i$ such that $\sum^n_{i=1}p_i \varphi_i = \tau_{\cM}$. 
%     Let $\rho_i\in \cM_+$ be the denisty operator of $\varphi_i$ with respect to $\tau_{\cM}$, and define $\mathcal{E}\in \mathbf{UCP}(\cM^{\oplus n},\cM)$ as
%     $\mathcal{E}(x_i)^n_{i=1} = \sum^n_{i=1}\rho^{1/2}_i x_i \rho^{1/2}_i$. 
%     Then we check that 
%     \begin{align*}
%         \sum^n_{i=1} p_i S(\varphi_i \circ\Phi,\varphi_i\circ \Psi) = S_{\tau_{\cM}}(\mathcal{E}\circ\Phi^{\oplus n},\mathcal{E}\circ\Psi^{\oplus n})
%     \end{align*}
% \end{remark}
\section{R\'{e}nyi Relative Entropies}
Given a finite von Neumann algebra $M$ with normal faithful normalized trace $\tau$, for any $1/2\leq p\leq \infty$ and any density operators $\rho,\sigma\in M$, the sandwiched R\'{e}nyi relative entropy between $\rho$ and $\sigma$ is defined as
\begin{align*}
    D_{p,\tau}(\rho\|\sigma)=\frac{1}{p-1}\log \tau_{\cM}(|\sigma^{-\frac{1}{2p'}} \rho \sigma^{-\frac{1}{2p'}}|^p).
\end{align*}
It is known that $D_{1,\tau}(\rho|\sigma)$ is the usual relative entropy, and as a function of $p$, $D_{p,\tau}(\rho\|\sigma)$ is non decreasing on $[1/2,+\infty]$. 
Moreover, it has been proved that R\'{e}nyi entropy is monotone under completly positive trace-preserving maps \cite{MDSFT2013}. 

Recall that for a finite inclusion $\cN\subseteq \cM$, the Pimsner-Popa index is defined as
\begin{align*}
    \lambda(\cM,\cN) = \sup \{\lambda>0\vert E^{\cM}_{\cN} - \lambda id_{\cM}\text{ is positive }\}.
\end{align*}
In \cite{GJL20}, Gao, Junge, and Laracuente studied relations between Pimsner-Popa index and R\'{e}nyi entropy. 
When $\cN\subset\cM$ is finite inclusion of tracial von Neumann algebras, they proved that (c.f. Proposition 3.2 in \cite{GJL20})
\begin{equation}\label{eqn:: result of Gao et.al.}
    -\log\lambda(\cM,\cN) \geq D_p(\cM|\cN)\geq H(\cM|\cN),\quad \forall p\in [1,+\infty],  
\end{equation}
where the quantity in the middle is defined as 
\begin{align*}
    D_p(\cM|\cN) = \sup_{\rho} \inf_{\sigma} D_p(\rho\|\sigma),
\end{align*}
with the supremum taken over all densities in $\cM$ and the infimum taken over all densities in $\cN$. 
Moreover, if $\cN\subset\cM$ are subfactors of type II$_1$ or finite dimensional then by Theorem 3.1 of \cite{GJL20} 
\begin{align*}
    -\log\lambda(\cM,\cN) = D_{p}(\cM|\cN),\quad \forall p\in[1/2,\infty].
\end{align*}

We consider the R\'{e}nyi relative entropy between completely positive bimodule maps. 
For a finite inclusion $\cN\subset\cM$, we define for $p\in [1/2,+\infty]$
\begin{align*}
    S_p(\Phi,\Psi) = D_{p,\tau_{\cM_2}} (\Delta^{1/2}\widehat{\Phi}\Delta^{1/2}\|\Delta^{1/2}\widehat{\Psi}\Delta^{1/2}), 
\end{align*}
where $\Phi\preccurlyeq \Psi $ are completely positive bimodule maps. 
\begin{definition}\label{def:: probabilistic constant for CP maps}
    Let $\mathcal{A}$ and $\mathcal{B}$ be von Neumann algebras. 
    For $\Phi,\Psi\in \mathbf{CP}(\mathcal{A},\mathcal{B})$ with $\Phi\preccurlyeq \Psi$, we define 
    \begin{align*}
        \lambda (\Phi,\Psi) = \sup\{\lambda>0\vert \Psi-\lambda\Phi\text{ is completely positive}\}. 
    \end{align*}
    Otherwise, set $\lambda (\Phi,\Psi) = +\infty$. 
\end{definition}
\begin{theorem}
    Let $\cN\subset\cM$ be a finite inclusion of finite von Neumann algebras. 
    Let $\Phi,\Psi\in \mathbf{CP}_{\cN}(\cM)$. 
    Then for $p\in[1,+\infty]$, 
    \begin{equation}\label{eqn:: compare to Gao et.al.}
        -\log \lambda(\Phi,\Psi) \geq S_p(\Phi,\Psi) \geq H(\Phi|\Psi).
    \end{equation}
\end{theorem}
\begin{proof}
    By Theorem \ref{thm::popaentropycompletepositive}, we have
    \begin{align*}
        H(\Phi|\Psi)\leq \delta D_{1,\tau_{\cM_2}}(\Delta^{1/2}\widehat{\Phi}\Delta^{1/2}\| \Delta^{1/2}\widehat{\Psi}\Delta^{1/2}).
    \end{align*}
    On the other hand, by definition
    \begin{align*}
        \delta D_{\infty,\tau_{\cM_2}}(\Delta^{1/2}\widehat{\Phi}\Delta^{1/2}\|\Delta^{1/2}\widehat{\Phi}\Delta^{1/2})= \log \inf\{\lambda>0\vert \lambda\widehat{\Psi}\geq \widehat{\Phi}\}.
    \end{align*}
    By Corollary \ref{corollary:: comparison of Fourier multiplier}, $\lambda\widehat{\Psi}-\widehat{\Phi}\geq 0$ if and only if $\lambda\Psi-\Phi$ is completely positive, 
    \begin{align*}
        \log \inf\{\lambda>0\vert \lambda\widehat{\Psi}\geq \widehat{\Phi}\} = -\log \lambda(\Phi,\Psi).
    \end{align*}
    So the result follows from that $D_{p,\tau_{\cM_2}}(\cdot|\cdot)$ is non decreasing with respect to $p$ in the interval $[1,\infty]$
\end{proof}
Let us take $\Phi = id_{\cM}$ and $\Psi = E_{\cN}$ as an example. 
When $\cN\subset\cM$ admits a downward Jones basic construction, for instance when $\cN\subset\cM$ is a subfactor, we have 
\begin{align*}
    S_{\infty}(id_{\cM},E_{\cN})= -\log\lambda(\cM,\cN) = \log [\cM:\cN],
\end{align*}
by Proposition \ref{prop:: norm of Fourier multiplier}, as well as 
\begin{align*}
    S_1(id_{\cM},E_{\cN}) = H(\cM|\cN)
\end{align*}
by Theorem \ref{thm::Formula for entropy between bimodule CP maps}. 
Thus in this case both bounds in Equation \eqref{eqn:: compare to Gao et.al.} are tight. 
Therefore the R\'{e}nyi relative entropy $S_p(id_{\cM},E_{\cN})$ interpolates between Pimsner-Popa index and Connes-S\o rmer relative entropy. 
This suggests that $S_p(\Phi,\Psi)$ is a more natural entropic quantity for completely positive maps. 

Now assuming $\cM$ is a finite factor, we can further compute, similar to the proof of Proposition \ref{prop:: explicit upperbound}, that
\begin{align*}
    S_{1/2}(id_{\cM},E_{\cN}) = 2\log \delta -\log \tau_{\cM}(\Delta^{-1}_0).
\end{align*}
We notice that if $\cN\subset\cM$ admits downward Jones basic construction, then 
\begin{align*}
     H(\cM|\cN) = S_1(id_{\cM},E_{\cN}) \geq S_{1/2}(id_{\cM},E_{\cN}),
\end{align*}
since the sandwiched R\'{e}nyi relative entropy does not decrease as the parameter increases. 
When the inclusion doesn't admit downward Jones basic construction, the reversed inequality can occur. 
For instance if $\cN = \bigoplus_{k\in K} M_{n_k}(\mathbb{C})$ and $\cM = M_{m}(\mathbb{C})$ with $m = \sum_{k\in K}a_kn_k$, then 
\begin{align*}
    S_{1}(id_{\cM},E_{\cN}) - S_{1/2}(id_{\cM},E_{\cN}) = \sum_{k\in K} \frac{n_ka_k}{m}\log \frac{a_k}{n_k}.
\end{align*}
Compare to 
\begin{align*}
    S_{1}(id_{\cM},E_{\cN}) - H(\cM|\cN) = \sum_{k\in K}\frac{n_ka_k}{m}\log\max\left\{\frac{a_k}{n_k},1\right\},
\end{align*}
we see that $S_{1/2}(id_{\cM},E_{\cN})>H(\cM|\cN)$ if and only if $a_k<n_k$ for some $k\in K$. 
Thus the sign of the difference $H(\cM|\cN)-S_{1/2}(id_{\cM},E_{\cN})$ can be treated as a criterion for the existence of downward Jones basic construction. 
\bibliography{Mainv4.bbl}
\end{document}